\documentclass[reqno]{amsart}
\usepackage{mathrsfs}
\usepackage{color}
\usepackage{amsmath}
\usepackage{amsfonts}
\usepackage{amssymb}
\usepackage{graphicx}
\usepackage{hyperref}

 \newtheorem{Theorem}{Theorem}[section]
 \newtheorem{Corollary}[Theorem]{Corollary}
 \newtheorem{Lemma}[Theorem]{Lemma}
 \newtheorem{Proposition}[Theorem]{Proposition}
 \newtheorem{Question}[Theorem]{Question}
 \newtheorem{Definition}[Theorem]{Definition}

 \newtheorem{Remark}[Theorem]{Remark}

 \numberwithin{equation}{section}

\begin{document}

\title[Boundary points, minimal $L^2$ integrals and concavity property \uppercase\expandafter{\romannumeral4}]
 {Boundary points, minimal $L^2$ integrals and concavity property \uppercase\expandafter{\romannumeral4}---fibrations over open Riemann surfaces}


\author{Qi'an Guan}
\address{Qi'an Guan: School of
Mathematical Sciences, Peking University, Beijing 100871, China.}
\email{guanqian@math.pku.edu.cn}

\author{Zhitong Mi}
\address{Zhitong Mi: Institute of Mathematics, Academy of Mathematics
and Systems Science, Chinese Academy of Sciences, Beijing, China
}
\email{zhitongmi@amss.ac.cn}

\author{Zheng Yuan}
\address{Zheng Yuan: School of
Mathematical Sciences, Peking University, Beijing 100871, China.}
\email{zyuan@pku.edu.cn}

\thanks{}



\date{\today}

\dedicatory{}

\commby{}


\begin{abstract}
In this article, we consider the minimal $L^2$ integrals related to modules at boundary points on fibrations over open Riemann surfaces,
and present a characterization for the concavity property of the minimal $L^2$ integrals degenerating to linearity.
\end{abstract}

\maketitle
\section{Introduction}
The strong openness property of multiplier ideal sheaves \cite{GZSOC}
(conjectured by Demailly \cite{DemaillySoc}) is an important feature of multiplier ideal sheaves and has opened the door to new types of approximation technique (see e.g. \cite{GZSOC,McNeal and Varolin,K16,cao17,cdM17,FoW18,DEL18,ZZ2018,GZ20,ZZ2019,ZhouZhu20siu's,FoW20,KS20,DEL21}),
where the multiplier ideal sheaf $\mathcal{I}(\varphi)$ is the sheaf of germs of holomorphic functions $f$ such that $|f|^2e^{-\varphi}$ is locally integrable (see e.g. \cite{Tian,Nadel,Siu96,DEL,DK01,DemaillySoc,DP03,Lazarsfeld,Siu05,Siu09,DemaillyAG,Guenancia}),
and $\varphi$ is a plurisubharmonic function on a complex manifold $M$ (see \cite{Demaillybook}).

Guan-Zhou \cite{GZSOC} proved the strong openness property (the 2-dimensional case was proved by Jonsson-Musta\c{t}\u{a} \cite{JonssonMustata}).
Using the strong openness property, Guan-Zhou \cite{GZeff} gave a proof of the following conjecture posed by Jonsson-Musta\c{t}\u{a} (see \cite{JonssonMustata}).

\textbf{Conjecture J-M}: If $c_o^F(\psi)<+\infty$, $\frac{1}{r^2}\mu(\{c_o^F(\psi)\psi-\log|F|<\log r\})$ has a uniform positive lower bound independent of $r\in(0,1)$, where $\mu$ is the Lebesgue measure on $\mathbb{C}^n$,
and $c_o^F(\psi):=\sup\{c\ge0:|F|^2e^{-2c\psi}$ is $L^1$ on a neighborhood of $o\}$ is the jumping number (see \cite{JonssonMustata}).
	
Recall that Jonsson-Musta\c{t}\u{a} \cite{JonssonMustata} posed Conjecture J-M, and proved the 2-dimensional case,
which deduced the 2-dimensional strong openness property. It is natural to ask: can one find a proof of Conjecture J-M independent of the strong openness property?

In \cite{BGY-boundary}, Bao-Guan-Yuan considered the minimal $L^2$ integrals related to modules at boundary points of the sublevel sets of plurisubharmonic functions on pseudoconvex domains,
and established a concavity property of the minimal $L^2$ integrals,
which deduced a proof of Conjecture J-M independent of the strong openness property.
In \cite{GMY-boundary2}, Guan-Mi-Yuan generalized the concavity property to weakly pseudoconvex K\"{a}hler manifolds.
	
	Note that the linearity is a degenerate concavity. It is natural to ask:	
	\begin{Question}
\label{Q:chara}
		How to characterize the concavity property  degenerating to linearity?
	\end{Question}
In \cite{GMY-boundary3}, Guan-Mi-Yuan  gave an answer to Question \ref{Q:chara} for the case of open Riemann surfaces.
	
In this article, we give an answer to Question \ref{Q:chara} for the case of fibrations over open Riemann surfaces.
\subsection{Main result}\label{sec:Main result}
Let $\Omega$  be an open Riemann surface, which admits a nontrivial Green function $G_{\Omega}$. Let $Y$ be an $n-1$ dimensional weakly pseudoconvex K\"ahler manifold, and let $K_Y$ be the canonical  line bundle on $Y$. Let $M=\Omega \times Y$ be an $n-$dimensional complex manifold. Let $\pi_{1}$ and $\pi_2$ be the natural projections from $M$ to $\Omega$ and $Y$ respectively. Let $K_M$ be the canonical line bundle on $M$.

  Let $\psi$ be a subharmonic function on $\Omega$. Let $\varphi_\Omega$ be a Lebesgue measurable function on $\Omega$ such that $\varphi_\Omega+\psi$ is subharmonic function on $\Omega$. Let $F$ be a holomorphic function on $\Omega$. Let $T\in [-\infty,+\infty)$.
Denote that
$$\tilde{\Psi}:=\min\{\psi-2\log|F|,-T\}.$$
For any $z \in \Omega$ satisfying $F(z)=0$,
we set $\tilde{\Psi}(z)=-T$. Denote that $\Psi:=\pi_1^*(\tilde\Psi)$ on $M$. Let $\varphi_Y$ be a plurisubharmonic function on $Y$. Denote that $\varphi:=\pi_1^*(\varphi_{\Omega})+\pi_2^*(\varphi_Y)$.

Let $p\in M$ be a point. Denote that $\tilde{J}(\Psi)_{p}:=\{f\in\mathcal{O}(\{\Psi<-t\}\cap V): t\in \mathbb{R}$ and $V$ is a neighborhood of $p\}$. We define an equivalence relation $\backsim$ on $\tilde{J}(\Psi)_{p}$ as follows: for any $f,g\in \tilde{J}(\Psi)_{p}$,
we call $f \backsim g$ if $f=g$ holds on $\{\Psi<-t\}\cap V$ for some $t\gg T$ and open neighborhood $V\ni p$.
Denote $\tilde{J}(\Psi)_{p}/\backsim$ by $J(\Psi)_{p}$, and denote the equivalence class including $f\in \tilde{J}(\Psi)_{p}$ by $f_{p}$.

If $p\in \cap_{t>T} \{\Psi<-t\}$, then $J(\Psi)_{p}=\mathcal{O}_{M,p}$ (the stalk of the sheaf $\mathcal{O}_{M}$ at $p$), and $f_{p}$ is the germ $(f,p)$ of holomorphic function $f$. If $p\notin \cap_{t>T} \overline{\{\Psi<-t\}}$, then  $J(\Psi)_{p}$ is trivial.

Let $f_{p},g_{p}\in J(\Psi)_{p}$ and $(h,p)\in \mathcal{O}_{M,p}$. We define $f_{p}+g_{p}:=(f+g)_{p}$ and $(h,p)\cdot f_{p}:=(hf)_{p}$.
Note that $(f+g)_{p}$ and $(hf)_{p}$ ($\in J(\Psi)_{p}$) are independent of the choices of the representatives of $f,g$ and $h$. Hence $J(\Psi)_{p}$ is an $\mathcal{O}_{M,p}$-module.

For $f_{p}\in J(\Psi)_{p}$ and $a,b\ge 0$, we call $f_{p}\in I\big(a\Psi+b\varphi\big)_{p}$ if there exist $t\gg T$ and a neighborhood $V$ of $p$,
such that $\int_{\{\Psi<-t\}\cap V}|f|^2e^{-a\Psi-b\varphi}dV_M<+\infty$, where $dV_M$ is a continuous volume form on $M$.
Note that $I\big(a\Psi+b\varphi\big)_{p}$ is an $\mathcal{O}_{M,p}$-submodule of $J(\Psi)_{p}$.

 Let ${Z}_0\subset M$ be a subset of $\cap_{t>T} \overline{\{\Psi<-t\}}$ such that there exists a subset $\tilde Z_0$ of $\Omega$ such that  $Z_0=\tilde{Z}_0\times Y$. Denote that $\tilde{Z}_1:=\{z\in \tilde{Z}_0:v(dd^c(\psi),z)\ge2ord_{z}(F)\}$ and $\tilde{Z}_2:=\{z\in \tilde{Z}_0:v(dd^c(\psi),z)<2ord_{z}(F)\},$ where $d^c=\frac{\partial-\bar\partial}{2\pi\sqrt{-1}}$ and $v(dd^c(\psi),z)$ is the Lelong number of $dd^c(\psi)$ at $z$ (see \cite{Demaillybook}).
  Denote that  $\tilde{Z}_3:=\{z\in \tilde{Z}_0:v(dd^c(\psi),z)>2ord_{z}(F)\}$. Note that $\{\tilde{\Psi}<-t\}\cup \tilde{Z}_3$ is an open Riemann surface for any $t\ge T$. Denote $Z_1:=\tilde{Z}_1\times Y$, $Z_2:=\tilde{Z}_2\times Y$ and $Z_3:=\tilde{Z}_3\times Y$ respectively.

Let $c(t)$ be a positive function on $(T,+\infty)$ such that $c(t)e^{-t}$ is decreasing on $(T,+\infty)$, $c(t)e^{-t}$ is integrable near $+\infty$, and $c(-\Psi)e^{-\varphi}$ has a positive lower bound on $K\cap\{\Psi<-T\}$ for any compact subset $K\subset M\backslash \pi_1^{-1}(E)$, where $E$  is an analytic subset of $\Omega$ such that $E\subset\{\tilde\Psi=-\infty\}$.

Let $f$ be a holomorphic $(n,0)$ form on $\{\Psi<-t_0\}\cap V$, where $V\supset Z_0$ is an open subset of $M$ and $t_0>T$
is a real number.
Denote
\begin{equation}
\label{def of g(t) for boundary pt}
\begin{split}
\inf\Bigg\{ \int_{ \{ \Psi<-t\}}|\tilde{f}|^2&e^{-\varphi}c(-\Psi): \tilde{f}\in
H^0(\{\Psi<-t\},\mathcal{O} (K_{M})  ) \\
&\&\, (\tilde{f}-f)_{p}\in
\mathcal{O} (K_{M})_{p} \otimes I(\varphi+\Psi)_{p}\text{ for any }  p\in Z_0 \Bigg\}
\end{split}
\end{equation}
by $G(t;c,\Psi,\varphi,I(\varphi+\Psi),f)$, where $t\in[T,+\infty)$ and $|f|^2:=(\sqrt{-1})^{n^2}f\wedge \bar{f}$ for any $(n,0)$ form $f$.
Without misunderstanding, we denote $G(t;c,\Psi,\varphi,I(\varphi+\Psi),f)$ by $G(t)$ for simplicity.

Recall that $G(h^{-1}(r))$ is concave with respect to $r$ (see \cite{GMY-boundary3}, see also Theorem \ref{thm:concavity}), where $h(t)=\int_t^{+\infty}c(s)e^{-s}ds$ for any $t\ge T$. In this article, we give a characterization of the concavity degenerating to linearity.

 Assume that $\tilde Z_3$ is  finite, and denote that $\tilde Z_3=\{z_1,z_2,\ldots,z_m\}$.  Let $w_j$ be a local coordinate on a neighborhood $V_{z_j}\Subset\Omega$ of $z_j$ satisfying $w_j(z_j)=0$ for any $j\in\{1,2,\ldots,m\}$, where $V_{z_j}\cap V_{z_k}=\emptyset$ for any $j\not=k.$
We give an answer to Question \ref{Q:chara} for the case of fibrations over open Riemann surfaces as follows.

\begin{Theorem}
	\label{thm:fibra-finite}
	For any $z\in \tilde Z_1$, assume that one of the following conditions holds:
	
	$(A)$ $\varphi_{\Omega}+a\psi$ is subharmonic near $z$ for some $a\in[0,1)$;
	
	$(B)$ $(\psi-2q_z\log|w|)(z)>-\infty$, where $q_z=\frac{1}{2}v(dd^c(\psi),z)$ and $w$ is a local coordinate on a neighborhood of $z$ satisfying that $w(z)=0$.
	
	If there exists $t_1\ge T$ such that $G(t_1)\in(0,+\infty)$, then $G(h^{-1}(r))$ is linear with respect to $r\in(0,\int_T^{+\infty}c(s)e^{-s}ds)$ if and only if the following statements hold:
	
	$(1)$ $f=\pi_1^*(a_jw_j^{k_j}dw_j)\wedge\pi_2^*(f_Y)+f_j$ on $(V_{z_j}\times Y)\cap\{\Psi<-t_0\}\cap V$ for any $j\in\{1,2,\ldots,m\}$, where $a_j\in\mathbb{C}\backslash\{0\}$, $k_j$ is a nonnegative integer, $f_Y$ is a holomorphic $(n-1,0)$ form on $Y$ satisfying $\int_Y|f_Y|^2e^{-\varphi_Y}\in(0,+\infty)$, and $(f_j)_p\in
\mathcal{O} (K_{\Omega})_{p} \otimes I(\varphi+\Psi)_{p}\text{ for any }  p\in {z_j}\times Y$;
	
	$(2)$ $\varphi_{\Omega}+\psi=2\log|g|+2\log|F|$, where $g$ is a holomorphic function on $\{\tilde\Psi<-T\}\cup \tilde Z_3\subset\Omega$ such that $ord_{z_j}(g)=k_j+1$ for any $j\in\{1,2,\ldots,m\}$;
		
	$(3)$ $\tilde Z_3\not=\emptyset$ and $\psi=2\sum_{1\le j\le m}\big(q_{z_j}-ord_{z_j}(F)\big)G_{\Omega_t}(\cdot,z_j)+2\log|F|-t$ on $\Omega_t$ for any $j\in\{1,2,\ldots,m\}$, where $\Omega_t=\{\tilde\Psi<-t\}\cup \tilde Z_3\subset\Omega$  and $G_{\Omega_t}$ is the Green function on $\Omega_t$;
	
	$(4)$ $\frac{q_{z_j}-ord_{z_j}(F)}{ord_{z_j}(g)}\lim_{z\rightarrow z_j}\frac{dg}{a_jw_j^{k_j}dw_j}=c_0$ for any $j\in\{1,2,\ldots,m\}$, where $c_0\in\mathbb{C}\backslash\{0\}$ is a constant independent of $j$.
\end{Theorem}

When $F\equiv 1$, $\psi(z)=-\infty$ for any $z\in \tilde Z_0=\tilde Z_3$ and condition $(B)$ holds, Theorem \ref{thm:fibra-finite} can be referred to \cite{BGY-concavity5} (see also Theorem \ref{finite-p} and Remark \ref{r:equivalent}).

When $\tilde Z_3$ is an infinite analytic subset of $\Omega$, Proposition \ref{p:n-linearity1} gives a necessary condition of $G(h^{-1}(r))$ is linear.

\section{Preparations}
\subsection{Properties of products of Bergman spaces}

In this section, we recall some results of products of Bergman spaces.

Let $U\subset \mathbb{C}^n$ and $W\subset \mathbb{C}^m$ be two open sets. Let $\varphi_1$ and  $\varphi_2$ be two  Lebesgue measurable functions on $U$ and $W$ respectively. Let $d\lambda_U$ and $d\lambda_W$ be the Lebesgue measures on $U$ and $W$ respectively. Denote $$A^2(U;e^{-\varphi_1}):=\{f\in\mathcal{O}_U:\int_U|f|^2e^{-\varphi_1}d\lambda_U<+\infty\},$$
$$A^2(W;e^{-\varphi_2}):=\{g\in\mathcal{O}_W:\int_W|g|^2e^{-\varphi_2}d\lambda_W<+\infty\}.$$
Denote $||f||_{1}:=(\int_U|f|^2e^{-\varphi_1}d\lambda_U)^{\frac{1}{2}}$ for any $f\in A^2(U;e^{-\varphi_1})$ and $||g||_{2}:=(\int_W|g|^2e^{-\varphi_2}d\lambda_W)^{\frac{1}{2}}$ for any $g\in A^2(W;e^{-\varphi_2})$ .
Let $M:=U\times W$ and $\varphi:=\varphi_1+\varphi_2$. Denote
$$A^2(M;e^{-\varphi}):=\{h\in\mathcal{O}_M:\int_M|h|^2e^{-\varphi}d\lambda_M<+\infty\}$$
and $||h||:=(\int_M|h|^2e^{-\varphi}d\lambda_M)^{\frac{1}{2}}$ for any $h\in A^2(M;e^{-\varphi})$.

We firstly recall the following lemma.
\begin{Lemma}[see \cite{Winiarski}]\label{module bounded by norm}
Let $\mu$ be a real positive Lebesgue measurable function on open subset $D_1\subset \mathbb{C}^n$. Assume that there exists a number $a>0$ such that the function $\mu^{-a}$ is integrable on some open set $D\subset D_1$ with respect to Lebesgue measure $d\lambda$. Let $f$ be a holomorphic function on $D$. Assume that $||f||_\mu<+\infty$, where  $||f||_\mu:=(\int_D|f|^2\mu d\lambda)^{\frac{1}{2}}$. Then for any compact set $K\subset D$, there exists a constant $C_K>0$ such that
$$\sup_K|f(z)|\le C_K||f||_\mu.$$

\end{Lemma}
\begin{proof}The following proof can be referred to \cite{Winiarski}. Let $K$ be a compact subset of $D$. There exists a real number $r>0$ (depends on $K$) such that $B(z,r)\subset D$ for any $z\in K$. Let $p:=\frac{1+a}{a}$ and $q=1+a$. Note that $|f|^{\frac{2}{p}}$ is plurisubharmonic function on $D$. It follows from sub-mean value inequality that for any $z\in K$, we have
$$|f(z)|^{\frac{2}{p}}\le \frac{1}{\text{Vol}(B(z,r))}\int_{B(z,r)}|f|^{\frac{2}{p}}d\lambda.$$
Then it follows from H\"{o}lder inequality that
\begin{equation}
\begin{split}
   \text{Vol}(B(z,r))|f(z)|^{\frac{2}{p}}& \le \int_{B(z,r)}|f|^{\frac{2}{p}}\mu^{\frac{1}{p}}\mu^{-\frac{1}{p}}d\lambda \\
     & \le \bigg( \int_{B(z,r)}|f|^{2}\mu d\lambda \bigg)^{\frac{1}{p}}
     \bigg( \int_{B(z,r)}\mu^{-\frac{q}{p}}d\lambda \bigg)^{\frac{1}{q}}\\
     &\le||f||_\mu^{\frac{2}{p}}\bigg( \int_{B(z,r)}\mu^{-a}d\lambda \bigg)^{\frac{1}{q}}.
\end{split}
\end{equation}
Hence
$$|f(z)|\le \bigg(\text{Vol}(B(z,r))\bigg)^{\frac{-p}{2}}\bigg( \int_{B(z,r)}\mu^{-a}d\lambda \bigg)^{\frac{p}{2q}}||f||_\mu.$$
Denote $C_K:=\bigg(\text{Vol}(B(z,r))\bigg)^{\frac{-p}{2}}\bigg(  \int_{B(z,r)}\mu^{-a}d\lambda \bigg)^{\frac{p}{2q}}$ and for any $z\in K$, we have
$$|f(z)|\le C_K||f||_\mu.$$

Lemma \ref{module bounded by norm} has been proved.
\end{proof}

\begin{Remark}\label{derivatives bounded by norm} Let $\alpha\in \mathbb{Z}^{n}_{\ge 0}$ be a multi-index. Let $f$ be a holomorphic function on $D$. For any compact subset $K\subset D$, it follows from Lemma \ref{module bounded by norm} and Cauchy integral formula that there exists a constant $C_{K,\alpha}>0$ such that  we have
$$\sup_K|\partial^{\alpha}f(z)|\le C_{K,\alpha}||f||_{\mu}.$$
\end{Remark}

In the following discussion, we assume that for any relatively compact set $U_1\Subset U$ ($W_2\Subset W$), there exists a real number $a_1>0$ ($a_2>0$) such that $e^{a_1\varphi_1}$ ($e^{a_2\varphi_2}$) is integrable on $U_1$ ($W_2$). Then we have the following proposition.

\begin{Proposition} $A^2(U;e^{-\varphi_1})$,  $A^2(W;e^{-\varphi_2})$ and  $A^2(M;e^{-\varphi})$ are separable Hilbert spaces.
\end{Proposition}
\begin{proof}
We prove that $A^2(U;e^{-\varphi_1})$ is separable Hilbert space. The same proof also holds for $A^2(W;e^{-\varphi_2})$ and  $A^2(M;e^{-\varphi})$.

It is clear that we only need to prove that $A^2(U;e^{-\varphi_1})$ is complete and separable.
 Let $U_1\Subset U$ be a relatively compact subset of $U$, then there exists $a_1>0$ such that $e^{a_1\varphi_1}$ is integrable on $U_1$. It follows from Lemma \ref{module bounded by norm} that we
\begin{equation}\label{hilbert space 1}\sup_{U_1}|f(z)|\le C_{U_1}||f||_1.
\end{equation}
Let $\{f_n\}$ be a Cauchy sequence in $A^2(U;e^{-\varphi_1})$. It follows from inequality \eqref{hilbert space 1} that $\{f_n\}$ is compactly convergent to a holomorphic function $F$ on $U$. Then
\begin{equation}\label{hilbert space 2}
\begin{split}
\int_U |f_n-F|^2e^{-\varphi_1}d\lambda_U&=\int_U \liminf_{k\to+\infty}|f_n-f_k|^2e^{-\varphi_1}d\lambda_U\\
&\le\liminf_{k\to+\infty}\int_U |f_n-f_k|^2e^{-\varphi_1}d\lambda_U\\
&=\liminf_{k\to+\infty}||f_n-f_k||_1.
\end{split}
\end{equation}
Hence when $n$ is large enough, we have $||F||_1\le ||f_n-F||_1+||f_n||\le \epsilon+||f_n||<+\infty$, i.e. $F\in A^2(U;e^{-\varphi_1})$.
It follows from $\{f_n\}$ is a Cauchy sequence in $A^2(U;e^{-\varphi_1})$ and \eqref{hilbert space 2} that we know that when $n$ is large enough, we have $\liminf_{k\to+\infty}||f_n-f_k||_1<\epsilon$, where $\epsilon>0$ is a small constant. Hence we have $\lim_{n\to+\infty}||f_n-F||_1=0$.  Hence the norm $||\cdot||_1$ is complete and $A^2(U;e^{-\varphi_1})$ is a Hilbert space.

Let $L^2(U;e^{-\varphi_1})$ be the space of the Lebesgue measurable function, which is square integrable with the weight $e^{-\varphi_1}$ on $U$. Note that $L^2(U;e^{-\varphi_1})$ is a separable Hilbert space (see \cite{Winiarski}). $A^2(U;e^{-\varphi_1})$ is a closed subspace of $L^2(U;e^{-\varphi_1})$, hence $A^2(U;e^{-\varphi_1})$ is  a separable Hilbert space.
\end{proof}

Next, we recall some results about products of Bergman spaces.

\begin{Lemma}\label{slice is integrable}Let $h\in A^2(M,e^{-\varphi})$. Let $\alpha\in \mathbb{Z}^{n}_{\ge 0}$ be a multi-index (let $\beta\in \mathbb{Z}^{m}_{\ge 0}$ be a multi-index). For any $z_0\in U$ ($w_0\in W$), we have $\partial^{\alpha}_z h(z_0,w)\in A^2(W,e^{-\varphi_2})$ ($\partial^{\beta}_w h(z,w_0)\in A^2(U,e^{-\varphi_1})$).
\end{Lemma}
\begin{proof}

 It follows from Remark \ref{derivatives bounded by norm} that there exists a constant $C_0>0$ ($C_0$ is independent of $w$) such that for any $w\in W$,
$$|\partial^{\alpha}_z h(z_0,w)|\le C_{0}||h(z,w)||_{1}.$$
Then it follows from $h\in A^2(M,e^{-\varphi})$ and Fubini's theorem that we have
\begin{equation}\nonumber
\begin{split}
   |\partial^{\alpha}_z h(z_0,w)||_2^2: & =\int_{W}|\partial^{\alpha}_z h(z_0,w)|^2e^{-\varphi_2}d\lambda_W \\
     &\le C_0 \int_{W}||h(z,w)||_{1}^2e^{-\varphi_2}d\lambda_W\\
     &=C_0\int_{W}\int_U |h(z,w)|^2e^{-\varphi_1-\varphi_2}d\lambda_M\\
     &=C_0||h||^2<+\infty.
\end{split}
\end{equation}
Hence we have $\partial^{\alpha}_z h(z_0,w)\in A^2(W,e^{-\varphi_2})$.

 The same proof as above shows that $\partial^{\alpha}_w h(z,w_0)\in A^2(U,e^{-\varphi_1})$.

Lemma \ref{slice is integrable} has been proved.
\end{proof}

The following lemma will be used in the proof of Lemma \ref{basis of product}.

\begin{Lemma}\label{holo of H}
Let $h\in A^2(M,e^{-\varphi})$ and $T\in A^2(U,e^{-\varphi_1})$. For any $w\in W$, denote
\begin{equation}\label{definition of H(w)}
  H(w):=\int_U h(z,w)\overline{T(z)}e^{-\varphi_1(z)}d\lambda_U.
\end{equation}
Then $H(w)$ is a holomorphic function on $W$ and $H(w)\in A^2(W,e^{-\varphi_2})$.
\end{Lemma}
\begin{proof}When $w\in W$ is fixed, it follows from Lemma \ref{slice is integrable} that $h(z,w)$ belongs to $A^2(U,e^{-\varphi_1})$. It follows from Cauchy-Schwarz inequality and both $T$ and $h(z,w)$ belong to $ A^2(U,e^{-\varphi_1})$ that $$|H(w)|\le \bigg(\int_U |h(z,w)|^2e^{-\varphi_1(z)}d\lambda_U\bigg)^{\frac{1}{2}}\bigg(\int_U |T(z)|^2e^{-\varphi_1(z)}d\lambda_U\bigg)^{\frac{1}{2}}<+\infty,$$ for any $w\in W$.

Let $w=(w_1,w_2,\ldots,w_m)\in W$ be given. Let $W_0\Subset W$ be an open convex neighborhood of $w$ in $W$. Let $\tilde{w}=(\tilde{w}_1,\tilde{w}_2,\ldots,\tilde{w}_m)\in W_0$. It follows from Remark \ref{derivatives bounded by norm} that we have
\begin{equation}\nonumber
\begin{split}
   &|h(z,\tilde{w})-h(z,w)-\sum_{j=1}^{m}(w_j-\tilde{w}_j)\partial_{w_j}h(z,w)|^2 \\
    =& |\int_{0}^{1}\frac{d h(z,w+t(\tilde{w}-w))}{dt}dt-\sum_{j=1}^{m}(w_j-\tilde{w}_j)\partial_{w_j}h(z,w)|^2\\
     =& |\sum_{j=1}^{m}(w_j-\tilde{w}_j)\int_{0}^{1}\partial_{w_j}h(z,w+t(\tilde{w}-w))dt
     -\sum_{j=1}^{m}(w_j-\tilde{w}_j)\partial_{w_j}h(z,w)|^2\\
     =&|\sum_{j=1}^{m}\sum_{k=1}^{m}(w_j-\tilde{w}_j)(w_k-\tilde{w}_k)
     \int_{0}^{1}\int_{0}^{1}t_1\big(\partial_{w_j}\partial_{w_k}h(z,w+t_1t_2(\tilde{w}-w))\big)dt_2dt_1
     |^2  \\
     \le &|\tilde{w}-w|^4 \frac{1}{4} \sum_{j=1}^{m}\sum_{k=1}^{m}\sup_{\hat{w}\in W_0}|\partial_{w_j}\partial_{w_k}(h(z,\hat{w}))|^2\\
     \le & C_1|\tilde{w}-w|^4||h(z,\cdot)||_2^2,
\end{split}
\end{equation}
where $C_1>0$ is a constant independent of $z$.
Then we have
\begin{equation}\label{complex differentiable of H}
\begin{split}
&|H(\tilde{w})-H(w)-\sum_{j=1}^{m}(w_j-\tilde{w}_j)\int_U (\partial_{w_j}h(z,w))\overline{T(z)}e^{-\varphi_1(z)}d\lambda_U|^2\\
=&|\int_U \bigg(h(z,\tilde{w})-h(z,w)-\sum_{j=1}^{m}(w_j-\tilde{w}_j)\partial_{w_j}h(z,w)\bigg)\overline{T(z)}e^{-\varphi_1(z)}d\lambda_U|^2\\
\le &||T||_1^2
\int_U |h(z,\tilde{w})-h(z,w)-\sum_{j=1}^{m}(w_j-\tilde{w}_j)\partial_{w_j}h(z,w)|^2e^{-\varphi_1(z)}d\lambda_U\\
\le &||T||_1^2
|\tilde{w}-w|^4 C_1  \int_U ||h(z,\cdot)||_2^2e^{-\varphi_1(z)}d\lambda_U\\
=&||T||_1^2
|\tilde{w}-w|^4 C_1 \int_U\int_W |h(z,\hat{w})|^2 e^{-\varphi_2(\hat{w})} e^{-\varphi_1(z)}d\lambda_Wd\lambda_U\\
=&C_1||T||_1^2 |\tilde{w}-{w}|^4 ||h||^2
\end{split}
\end{equation}
It follows from inequality \eqref{complex differentiable of H} that $H(w)$ is a holomorphic function on $W$. Note that

\begin{equation}\label{norm of H 1}
\begin{split}
  ||H(w)||_2^2
  &\le \int_W|H(w)|^2e^{-\varphi_2}d\lambda_W\\
  &\le \bigg(\int_U |T(z)|^2e^{-\varphi_1(z)}d\lambda_U\bigg)\int_W \int_U |h(z,w)|^2e^{-\varphi_1(z)-\varphi_2(w)}d\lambda_Ud\lambda_W\\
  &=||T||_1^2\cdot ||h||^2<+\infty.
\end{split}
\end{equation}
Hence we have $H(w)\in A^2(W,e^{-\varphi_2})$.
\end{proof}

The following lemma implies that the product of bases of $A^2(U;e^{-\varphi_1})$ and $A^2(W;e^{-\varphi_2})$ make a basis of $A^2(M;e^{-\varphi})$.

\begin{Lemma}\label{basis of product}Let $\{f_i(z)\}_{i\in \mathbb{Z}_{\ge 0}}$ and $\{g_j(w)\}_{j\in \mathbb{Z}_{\ge 0}}$ be the complete orthonormal bases of $A^2(U,e^{-\varphi_1})$ and $A^2(W,e^{-\varphi_2})$ respectively. Then $\{f_i(z) g_j(w)\}_{i,j\in\mathbb{Z}_{\ge 0}}$ is a complete orthonormal basis of $A^2(M,e^{-\varphi})$.
\end{Lemma}

\begin{proof}It follows from Fubini's theorem that $\{f_i(z) g_j(w)\}_{i,j\in\mathbb{Z}_{\ge 0}}$ is orthonormal basis of $A^2(M,e^{-\varphi})$. Now we prove $\{f_i(z) g_j(w)\}_{i,j\in\mathbb{Z}_{\ge 0}}$  is complete. Let $h(z,w)\in A^2(M,e^{-\varphi})$ such that $\int_M h(z,w)\overline{f_i(z)g_j(w)}e^{-\varphi}d\lambda_M=0$ for any $i,j\in \mathbb{Z}_{\ge 0}$.

For any $i_0\in \mathbb{Z}_{\ge 0}$, denote
$$H_{i_0}(w):=\int_{U}h(z,w)\overline{f_{i_0}(z)}e^{-\varphi_1}d\lambda_U,$$
where $w\in W$. It follows from Lemma \ref{holo of H} that $H_{i_0}(w)\in A^2(W,e^{-\varphi_2})$. It follows from Fubini's theorem that for any $j_0\in \mathbb{Z}_{\ge 0}$, we have
\begin{equation}
  \begin{split}
     0 &=\int_M h(z,w)\overline{f_{i_0}(z)g_{j_0}(w)}e^{-\varphi}d\lambda_M \\
       &= \int_W \big(\int_U h(z,w)\overline{f_{i_0}(z)}e^{-\varphi_1}d\lambda_U\big)\overline{g_{j_0}(w)}e^{-\varphi_2}d\lambda_W\\
       &=\int_W H_{i_0}(w)\overline{g_{j_0}(w)}e^{-\varphi_2}d\lambda_W.
  \end{split}
\end{equation}
As $j_0$ is arbitrarily chosen and $\{g_j(w)\}_{j\in \mathbb{Z}_{\ge 0}}$ is the complete orthonormal basis of $A^2(W,e^{-\varphi_2})$, we know that $H_{i_0}(w)= 0$ for any $i_0\in \mathbb{Z}_{\ge 0}$.

For any $w\in W$, it follows from $H_{i}(w)= 0$ for any $i\in \mathbb{Z}_{\ge 0}$ and $\{f_i(z)\}_{i\in \mathbb{Z}_{\ge 0}}$ is the complete orthonormal basis of $A^2(U,e^{-\varphi_1})$, we know that $h(z,w)= 0$, for any $z\in U$. Since $w$ is arbitrarily chosen, we know that $h(z,w)\equiv 0$ on $M$, for any $z\in U$. This shows that $\{f_i(z) g_j(w)\}_{i,j\in\mathbb{Z}_{\ge 0}}$ is complete.

Lemma \ref{basis of product} has been proved.
\end{proof}

Let $\Delta\subset\mathbb{C}$  be  the unit disk, and let $Y$ be an $(n-1)-$dimensional complex manifold, and let $N=\Delta \times Y$. Let $\pi_1$ and $\pi_2$ be the natural projections from $N$ to $\Delta$ and $Y$ respectively.  Let $\rho_1$ be a nonnegative Lebesgue measurable function on $\Delta$ satisfying that $\rho_1(w)=\rho_1(|w|)$ for any $w$ and the Lebesgue measure of $\{w\in\Delta: \rho_1(w)>0\}$ is  positive. Let $\rho_2$ be a nonnegative Lebesgue measurable function on $Y$, and denote that $\rho=\pi_1^*(\rho_1)\times\pi_2^*(\rho_2)$ on $N$. We recall the following lemma.

\begin{Lemma}[see \cite{BGY-concavity6}]
	\label{decomp}	For any holomorphic $(n,0)$  form $F$ on $N$, there exists a unique sequence of  holomorphic  $(n-1,0)$   forms $\{F_{j}\}_{j\in\mathbb{Z}_{\ge0}}$ on $Y$ such that
	\begin{equation*}
		F=\sum_{j\in\mathbb{Z}_{\ge0}}\pi_1^*(w^{j}dw_j) \wedge \pi_2^*(F_j),
	\end{equation*}
where the right term of the above equality is uniformly convergent on any compact subset of $N$. Moreover, if 	$\int_{N}|F|^2\rho<+\infty,$ we have
	\begin{equation*}
\int_{Y}|F_{j}|^2\rho_2<+\infty
\end{equation*}
for any $j\in\mathbb{Z}_{\ge0}$.
\end{Lemma}

In the following discussion, we assume that
 $U\Subset \mathbb{C}$ is an open subset containing origin $0$ in $\mathbb{C}$.

\begin{Lemma}\label{construction of basis in dim one}There exists a countable complete orthonormal basis $\{\tilde{f}_i(z)\}_{i\in \mathbb{Z}_{\ge 0}}$  of $A^2(U,e^{-\varphi_1})$ such that $k_i:=\text{ord}_0(\tilde{f}_i)$ is strictly increasing with respect to $i$.
\end{Lemma}
\begin{proof}For any $k\in\mathbb{Z}_{\ge 0}$, denote $$A^2_k(U,e^{-\varphi_1}):=\{f\in A^2(U,e^{-\varphi_1}): (\partial^k_zf)(0)=1\  \& \ (\partial^j_zf)(0)=0 \text{ for } j<k\}.$$
If $A^2_k(U,e^{-\varphi_1})\neq \emptyset$, denote $S_k:=\inf\{||f||_1: f\in A^2_k(U,e^{-\varphi_1})\}$. We prove that there exists a $F_k\in A^2_k(U,e^{-\varphi_1})$ such that $S_k=||F_k||_1$.

Let $\{F_j\}_{j\in\mathbb{Z}_{\ge 0}}$ be a sequence of holomorphic functions in $A^2_k(U,e^{-\varphi_1})$ such that $||F_j||_1\to S_k$ as $j\to+\infty$. It follows from Lemma \ref{module bounded by norm} that for any relative compact subset $U_1\Subset U$, we have $\sup_{U_1}|F_j|\le C_1 ||F_j||_1$ for some $C_1>0$. As $||F_j||_1\to S_k$ when $j\to+\infty$, we know that $\sup_{j}\sup_{U_1}|F_j|<C_2$ for some $C_2>0$. Hence there exists a subsequence of $\{F_j\}_{j\in\mathbb{Z}_{\ge 0}}$ (also denote by $\{F_j\}_{j\in\mathbb{Z}_{\ge 0}}$) compactly convergent to a holomorphic function $F_k$ on $U$. As $\{F_j\}_{j\in\mathbb{Z}_{\ge 0}}$ is compactly convergent to $F_k$ and $F_j\in A^2_k(U,e^{-\varphi_1})$, we know that $(\partial^k_zF_k)(0)=1$ and $(\partial^j_zF_k)(0)=0 \text{ for } j<k$. It follows from Fatou's Lemma that we know
$$\int_U |F_k|^2e^{-\varphi_1}d\lambda_U\le \liminf_{j\to+\infty} \int_U |F_j|^2e^{-\varphi_1}d\lambda_U=S_k,$$
which implies that $F_k\in A^2_k(U,e^{-\varphi_1})$ and $||F_k||_1\le S_k$. By definition of $S_k$, we know that $||F_k||_1= S_k$.

Denote $\mathbb{K}:=\{k\in \mathbb{Z}_{\ge 0}: A^2_k(U,e^{-\varphi_1})\neq \emptyset\}$. Note that $\mathbb{K}$ is a countable infinite subset of $\mathbb{Z}_{\ge 0}$.
For any $k\in \mathbb{K}$,  we take $f_k:=\frac{F_k}{||F_k||_1}$.

 Now we prove $\{f_k(z)\}_{k\in \mathbb{K}}$ is a complete orthonormal basis of $A^2(U,e^{-\varphi_1})$.

Firstly, we prove that $\{f_k(z)\}_{k\in \mathbb{K}}$ is orthonormal basis.
Let $k\in \mathbb{K}$. We prove that for any $g\in A^2(U,e^{-\varphi_1})$ satisfying $(\partial^j_z g)(0)=0$ for any $j\le k$, we have $\langle f_k,g\rangle_1=0$, where $\langle\cdot,\cdot\rangle_1$ is the inner product on $A^2(U,e^{-\varphi_1})$. Note that for any $\alpha\in \mathbb{C}$, we have $f_k+\alpha g\in A^2_k(U,e^{-\varphi_1})$. Hence for any $\alpha\in \mathbb{C}$, $||f_k+\alpha g||_1\ge S_k$. However let $\alpha=-\frac{\langle f_k,g\rangle_1}{||g||_1^2}$, then by direct computation we have
$$||f_k+\alpha g||^2_1=||f_k||_1^2-|\frac{\langle f_k,g\rangle_1}{||g||_1^2}|^2.$$
If $\langle f_k,g\rangle_1\neq0$, we will have $||f_k+\alpha g||_1<S_k$, which is a contradiction. Hence we have $\langle f_k,g\rangle_1=0$. By the above discussion, we also know that $f_k$ is the unique holomorphic function in $A^2_k(U,e^{-\varphi_1})$ such that  $||f_k||_1= S_k$.
 For any $k_1<k_2\in \mathbb{K}$, following above discussion, we have $\langle f_{k_1},f_{k_2}\rangle_1=0$. By definition, we also have $||f_{k}||_1=1$. Hence $\{f_k(z)\}_{k\in \mathbb{K}}$ is orthonormal basis.

Secondly, we prove that $\{f_k(z)\}_{k\in \mathbb{K}}$ is complete. If $k\in \mathbb{Z}_{\ge 0}\backslash\mathbb{K}$, then we set $f_k(z)\equiv 0$. Let $h\in A^2(U,e^{-\varphi_1})$ be a holomorphic function. Let $\{a_k\}_{k\in\mathbb{Z}_{\ge 0}}$ be a sequence of complex numbers which will be determined later. Let
$$h_k(z)=\sum_{j\le k}a_jf_j(z).$$
Now we will choose $\{a_k\}_{k\in\mathbb{Z}_{\ge 0}}$ such that for any $j\le k$,
\begin{equation}\label{determine coeff}
(\partial^j_zh_k)(0)=(\partial^j_zh)(0).
\end{equation}

If $0\in \mathbb{Z}_{\ge 0}\backslash\mathbb{K}$, we set $a_0=0$. If $0\notin \mathbb{Z}_{\ge 0}\backslash\mathbb{K}$, we set $a_0=h(0)||F_0||_1$. By the construction of $\{f_i(z)\}_{i\in \mathbb{Z}_{\ge 0}}$, we know that $h_k(0)=h(0)$. Assume that for any $j\le k-1$, the sequence $\{a_j\}_{0\le j\le k-1}$ has been chosen. Now we choose the number $a_k$. If $k\in \mathbb{Z}_{\ge 0}\backslash\mathbb{K}$, then
$A^2_k(U,e^{-\varphi_1})=\emptyset$. We also note that $h-h_{k-1}$ satisfies $$\bigg(\partial^j_z(h-h_{k-1})\bigg)(0)=0,\text{ for any } j\le k-1.$$
Then we have $\bigg(\partial^k_z(h-h_{k-1})\bigg)(0)=0$, otherwise $\frac{(h(z)-h_{k-1}(z))}{(\partial^k_z(h-h_{k-1}))(0)}\in A^2_k(U,e^{-\varphi_1})$ which is a contradiction. Since $\bigg(\partial^k_z(h-h_{k-1})\bigg)(0)=0$, we set $a_k=0$. If $k\in \mathbb{K}$, it follows from
 $(\partial^k_zh_k)(0)=(\partial^k_zh)(0)$ that we have the following equation
 \begin{equation}\label{equation of coeff}
 a_0(\partial^k_zf_0)(0)+a_1(\partial^k_zf_1)(0)+\cdots+a_k(\partial^k_zf_k)(0)=(\partial^k_zh)(0).
 \end{equation}
Note that $(\partial^k_zf_k)(0)=\frac{1}{||F_k||_1}\neq 0$ and $\{a_j\}_{0\le j\le k-1}$ has been chosen. The equation \eqref{equation of coeff} can be solved, i.e. we can find $a_k$ such that $(\partial^k_zh_k)(0)=(\partial^k_zh)(0)$ holds. Hence we can choose $\{a_k\}_{k\in\mathbb{Z}_{\ge 0}}$ by induction
 such that for any $j\le k$,
\begin{equation}\nonumber
(\partial^j_zh_k)(0)=(\partial^j_zh)(0).
\end{equation}

Note that for any $j\le k$, we have $\bigg(\partial^j_z(h_k-h)\bigg)(0)=0$. Hence we know that
$\langle h_k-h,f_k\rangle_1=0$, i.e. $$\langle h,f_k\rangle_1=\langle h_k,f_k\rangle_1=a_k.$$
By Bessel inequality, we have $\sum_{j=0}^{+\infty}|a_j|^2\le ||h||_1^2$. Denote $H(z):=\sum_{j=0}^{+\infty}a_jf_j(z)$, where the right hand-side is uniformly convergent to $H(z)$ on any compact subset of $U$ and we have $H(z)\in A^2(U,e^{-\varphi_1})$. Hence for any $k\in \mathbb{Z}_{\ge 0}$, we have
\begin{equation}\nonumber
(\partial^k_zH)(0)=\lim_{j\to +\infty}(\partial^k_zh_k)(0)=(\partial^k_zh)(0).
\end{equation}
Hence $h\equiv H$, which implies that $\{f_k(z)\}_{k\in \mathbb{K}}$ is complete.

Denote $\tilde{f}_1:=f_{k_1}$, where $k_1$ satisfies $ord_0(f_{k_1})$ is minimal among $k\in\mathbb{K}$. Denote
$\tilde{f}_2:=f_{k_2}$, where $k_2$ satisfies $ord_0(f_{k_2})$ is minimal among $k\in\mathbb{K}\backslash\{k_1\}$. Denote $\tilde{f}_i:=f_{k_i}$, where $k_i$ satisfies $ord_0(f_{k_i})$ is minimal among $k\in\mathbb{K}\backslash\{k_1,k_2,\ldots,k_{i-1}\}$.  Now  $\{ord_0(\tilde{f}_i)\}_{i\in\mathbb{Z}_{\ge 0}}$ is strictly increasing with respect to $i$. Note that  $\{\tilde{f}_i(z)\}_{i\in \mathbb{Z}_{\ge 0}}$ is a rearrangement of $\{f_k(z)\}_{k\in \mathbb{K}}$ and $\{f_k(z)\}_{k\in \mathbb{K}}$ is a complete orthonormal basis of $A^2(U,e^{-\varphi_1})$. We know that $\{\tilde{f}_i(z)\}_{i\in \mathbb{Z}_{\ge 0}}$ satisfies the requirement of Lemma \ref{construction of basis in dim one}. We are done.

\end{proof}

In the following discussion, let
 $U=\Delta\subset \mathbb{C}$ be the unit disk. It follows from Lemma \ref{construction of basis in dim one} that there exists a complete orthonormal basis $\{\tilde{f}_i(z)\}_{i\in \mathbb{Z}_{\ge 0}}$ of $A^2(U,e^{-\varphi_1})$ which satisfies $\text{ord}_0(\tilde{f}_i)$ is strictly increasing with respect to $i$.  Let $W\subset \mathbb{C}^{m}$ be an open set.
Let $\{g_j(w)\}_{j\in \mathbb{Z}_{\ge 0}}$ be the complete orthonormal basis of $A^2(W,e^{-\varphi_2})$.
Let $F$ be a holomorphic function on $M:=U\times W$ satisfying that $F\in A^2(M,e^{-\varphi})$. It follows from Lemma \ref{decomp} that $F=\sum_{l\ge \tilde{k}} z^lF_l(w)$, where $\{F_l\}_{l\ge \tilde{k}}$ is a sequence of holomorphic functions on $W$ and $F_{\tilde{k}}(w)\not\equiv 0$. Denote $k_i:=\text{ord}_0(\tilde{f}_i)$ for any $i\in\mathbb{Z}_{\ge 0}$ and $k:=\inf_{i\in\mathbb{Z}_{\ge 0}}\{k_i\}=k_1$.

\begin{Lemma}\label{decomp integrable} We have $\tilde{k}\ge k$ and for any $l\ge \tilde{k}$, $F_l\in A^2(W,e^{-\varphi_2})$.
\end{Lemma}

\begin{proof}
It follows from Lemma \ref{basis of product} that $\{\tilde{f}_i(z) g_j(w)\}_{i,j\in\mathbb{Z}_{\ge 0}}$ is a complete orthonormal basis of $A^2(M,e^{-\varphi})$.

As $F\in A^2(M,e^{-\varphi})$, we know that
$$F(z,w)=\sum_{i,j\in\mathbb{Z}_{\ge 0}}a_{ij}\tilde{f}_i(z) g_j(w)=\sum_{i\in \mathbb{Z}_{\ge 0}}\tilde{f}_i(z)\sum_{j\in\mathbb{Z}_{\ge 0}}a_{ij}g_j(w),$$ for some $a_{ij}\in \mathbb{C}$, where the right-hand side  is uniformly convergent on any compact subset of $M$. It follows from $F\in A^2(M,e^{-\varphi})$ that we know $\sum_{j\in\mathbb{Z}_{\ge 0}}a_{ij}g_j(w)\in A^2(W,e^{-\varphi_2})$ for any $i\in\mathbb{Z}_{\ge 0}$.

As $U\subset \mathbb{C}$ is unit disc, we assume that for each $i\in\mathbb{Z}_{\ge 0}$, $\tilde{f}_i(z)=\sum_{l\ge k_i}^{+\infty}b_{il}z^l$ on $U$ for some $b_{il}\in \mathbb{C}$, where the right-hand side is uniformly convergent on any compact subset of $U$. Then we know that $F(z,w)=\sum_{l\ge k}z^l\bigg(\sum_{i\in \mathbb{I}_l}b_{il}(\sum_{j\in\mathbb{Z}_{\ge 0}}a_{ij}g_j(w))\bigg)$, where the index set $\mathbb{I}_l:=\{i\in\mathbb{Z}_{\ge 0}:k_i\le l\}$ and the right-hand side  is uniformly convergent on any compact subset of $M$. For fixed $l\ge k$, as $k_i\in \mathbb{Z}_{\ge 0}$ is strictly increasing with respect to $i$, we know that $\mathbb{I}_l$ is a finite set for any $l\ge k$.

Recall that it follows from Lemma \ref{decomp} that $F=\sum_{l\ge \tilde{k}} z^lF_l(w)$ and $\{F_l\}_{l\ge \tilde{k}}$ is unique. The uniqueness of $\{F_l\}_{l\ge \tilde{k}}$ implies that we have $\sum_{i\in \mathbb{I}_l}b_{il}(\sum_{j\in\mathbb{Z}_{\ge 0}}a_{ij}g_j(w))=0$ for $k\le l<\tilde{k}$ and $F_l(w)=\sum_{i\in \mathbb{I}_l}b_{il}(\sum_{j\in\mathbb{Z}_{\ge 0}}a_{ij}g_j(w))$ for any $l\ge \tilde{k}$. It follows from $\mathbb{I}_l$ is a finite set for any $l\ge k$ and $\sum_{j\in\mathbb{Z}_{\ge 0}}a_{ij}g_j(w)\in A^2(W,e^{-\varphi_2})$ for any $i\in\mathbb{Z}_{\ge 0}$. that we have $F_l\in A^2(W,e^{-\varphi_2})$ for any $l\ge \tilde{k}$.
\end{proof}

\begin{Remark}\label{remark of decomp integrable} By the definition of $k$, we know that for any $l\ge k$, $|z|^{2l}e^{-\varphi_1}$ is $L^1$ integrable near $0\in\mathbb{C}$. As $\tilde{k}\ge k$, we have $|z|^{2l}e^{-\varphi_1}$ is $L^1$ integrable near $0\in\mathbb{C}$ for any $l\ge \tilde{k}$.
\end{Remark}

\subsection{Concavity property on weakly pseudoconvex K\"ahler manifolds} In this section, we recall some results about the concavity property on weakly pseudoconvex K\"ahler manifolds \cite{GMY-boundary3} (see also \cite{GMY-boundary2}).
Let $M$ be a complex manifold. Let $X$ and $Z$ be closed subsets of $M$. We call that a triple $(M,X,Z)$ satisfies condition $(A)$, if the following two statements hold:

$\uppercase\expandafter{\romannumeral1}.$ $X$ is a closed subset of $M$ and $X$ is locally negligible with respect to $L^2$ holomorphic functions; i.e., for any local coordinated neighborhood $U\subset M$ and for any $L^2$ holomorphic function $f$ on $U\backslash X$, there exists an $L^2$ holomorphic function $\tilde{f}$ on $U$ such that $\tilde{f}|_{U\backslash X}=f$ with the same $L^2$ norm;

$\uppercase\expandafter{\romannumeral2}.$ $Z$ is an analytic subset of $M$ and $M\backslash (X\cup Z)$ is a weakly pseudoconvex K\"ahler manifold.

Let $M$ be an $n-$dimensional complex manifold.
Assume that $(M,X,Z)$  satisfies condition $(A)$. Let $K_M$ be the canonical line bundle on $M$.
Let $dV_M$ be a continuous volume form on $M$. Let $F$ be a holomorphic function on $M$.
Assume that $F$ is not identically zero. Let $\psi$ be a plurisubharmonic function on $M$.
Let $\varphi$ be a Lebesgue measurable function on $M$ such that $\varphi+\psi$ is a plurisubharmonic function on $M$.

Let $T\in [-\infty,+\infty)$.
Denote that
$$\Psi:=\min\{\psi-2\log|F|,-T\}.$$
For any $z \in M$ satisfying $F(z)=0$,
we set $\Psi(z)=-T$.

\begin{Definition}[\cite{GMY-boundary3}]
We call that a positive measurable function $c$  on $(T,+\infty)$ is in class $\tilde{P}_{T,M,\Psi}$ if the following two statements hold:
\par
$(1)$ $c(t)e^{-t}$ is decreasing with respect to $t$;
\par
$(2)$ For any $t_0> T$, there exists a closed subset $E_0$ of $M$ such that $E_0\subset Z\cap \{\Psi(z)=-\infty\}$ and for any compact subset $K\subset M\backslash E_0$, $e^{-\varphi}c(-\Psi)$ has a positive lower bound on $K\cap \{\Psi<-t_0\}$ .
\end{Definition}

For $f_{z_0}\in J(\Psi)_{z_0}$ and $a,b\ge 0$, we call $f_{z_0}\in I\big(a\Psi+b\varphi\big)_{z_0}$ if there exist $t\gg T$ and a neighborhood $V$ of $z_0$,
such that $\int_{\{\Psi<-t\}\cap V}|f|^2e^{-a\Psi-b\varphi}dV_M<+\infty$.
Note that $I\big(a\Psi+b\varphi\big)_{z_0}$ is an $\mathcal{O}_{M,z_0}$-submodule of $J(\Psi)_{z_0}$.

Let $Z_0$ be a subset of $\cap_{t>T} \overline{\{\Psi<-t\}}$. Let $f$ be a holomorphic $(n,0)$ form on $\{\Psi<-t_0\}\cap V$, where $V\supset Z_0$ is an open subset of $M$ and $t_0\ge T$
is a real number.
Let $J_{z_0}$ be an $\mathcal{O}_{M,z_0}$-submodule of $J(\Psi)_{z_0}$ such that $I\big(\Psi+\varphi\big)_{z_0}\subset J_{z_0}$,
where $z_0\in Z_0$.
Let $J$ be the $\mathcal{O}_{M,z_0}$-module sheaf with stalks $J_{z_0}$, where $z_0\in Z_0$.
Denote the minimal $L^{2}$ integral related to $J$
\begin{equation}
\label{def of g(t) for boundary pt}
\begin{split}
\inf\Bigg\{ \int_{ \{ \Psi<-t\}}|\tilde{f}|^2e^{-\varphi}c(-\Psi)&: \tilde{f}\in
H^0(\{\Psi<-t\},\mathcal{O} (K_M)  ) \\
&\&\, (\tilde{f}-f)_{z_0}\in
\mathcal{O} (K_M)_{z_0} \otimes J_{z_0},\text{ for any }  z_0\in Z_0 \Bigg\}
\end{split}
\end{equation}
by $G(t;c,\Psi,\varphi,J,f)$, where $t\in[T,+\infty)$, $c$ is a nonnegative function on $(T,+\infty)$ and $|f|^2:=\sqrt{-1}^{n^2}f\wedge \bar{f}$ for any $(n,0)$ form $f$.
Without misunderstanding, we denote $G(t;c,\Psi,\varphi,J,f)$ by $G(t)$ for simplicity. For various $c(t)$, we denote $G(t;c,\Psi,\varphi,J,f)$ by $G(t;c)$ respectively for simplicity.

We recall the concavity property for $G(t)$.

\begin{Theorem}[\cite{GMY-boundary3}]\label{thm:concavity}
Let $c\in\tilde{P}_{T,M,\Psi}$ satisfying that $\int_{T_1}^{+\infty}c(s)e^{-s}ds<+\infty$, where $T_1>T$. If there exists $t \in [T,+\infty)$ satisfying that $G(t)<+\infty$, then $G(h^{-1}(r))$ is concave with respect to  $r\in (0,\int_{T}^{+\infty}c(t)e^{-t}dt)$, $\lim\limits_{t\to T+0}G(t)=G(T)$ and $\lim\limits_{t \to +\infty}G(t)=0$, where $h(t)=\int_{t}^{+\infty}c(s)e^{-s}ds$.
\end{Theorem}

When $F\equiv1$ and $\psi(z)=-\infty$ for any $z\in Z_0$, $J_{z}$ is an ideal of $\mathcal{O}_{M,z}$ for any $z\in Z_0$ and Theorem \ref{thm:concavity} degenerates to the concavity property  with respect to the ideals at inner points (\cite{GMY-boundary2}, see also \cite{GMY-concavity2}).

Let $c(t)$ be a nonnegative measurable function on $(T,+\infty)$. Denote that
\begin{equation}\nonumber
\begin{split}
\mathcal{H}^2(t;c):=\Bigg\{\tilde{f}:\int_{ \{ \Psi<-t\}}|\tilde{f}|^2e^{-\varphi}c(-\Psi)<+\infty,\  \tilde{f}\in
H^0(\{\Psi<-t\},\mathcal{O} (K_M)  ) \\
\& (\tilde{f}-f)_{z_0}\in
\mathcal{O} (K_M)_{z_0} \otimes J_{z_0},\text{ for any }  z_0\in Z_0  \Bigg\},
\end{split}
\end{equation}
where $t\in[T,+\infty)$.

The following corollary gives a necessary condition for the concavity  degenerating to linearity.

\begin{Corollary}[\cite{GMY-boundary3}]\label{c:necessary condition for linear of G}
Let $c\in\tilde{P}_{T,M,\Psi}$ satisfying that $\int_{T_1}^{+\infty}c(s)e^{-s}ds<+\infty$, where $T_1>T$.
Assume that $G(t)\in(0,+\infty)$ for some $t\ge T$, and $G({h}^{-1}(r))$ is linear with respect to $r\in[0,\int_T^{+\infty}c(s)e^{-s}ds)$, where ${h}(t)=\int_{t}^{+\infty}c(s)e^{-s}ds$.

Then there exists a unique holomorphic $(n,0)$ form $\tilde{F}$ on $\{\Psi<-T\}$
such that $(\tilde{F}-f)_{z_0}\in\mathcal{O} (K_M)_{z_0} \otimes J_{z_0}$ holds for any  $z_0\in Z_0$,
and $G(t)=\int_{\{\Psi<-t\}}|\tilde{F}|^2e^{-\varphi}c(-\Psi)$ holds for any $t\ge T$.

Furthermore
\begin{equation*}
\begin{split}
  \int_{\{-t_1\le\Psi<-t_2\}}|\tilde{F}|^2e^{-\varphi}a(-\Psi)=\frac{G(T_1;c)}{\int_{T_1}^{+\infty}c(t)e^{-t}dt}
  \int_{t_2}^{t_1}a(t)e^{-t}dt
\end{split}
\end{equation*}
holds for any nonnegative measurable function $a$ on $(T,+\infty)$, where $T\le t_2<t_1\le+\infty$ and $T_1 \in (T,+\infty)$.
\end{Corollary}

\begin{Remark}[\cite{GMY-boundary3}]
\label{rem:linear}
If $\mathcal{H}^2(t_0;\tilde{c})\subset\mathcal{H}^2(t_0;c)$ for some $t_0\ge T$, we have
\begin{equation*}
\begin{split}
  G(t_0;\tilde{c})=\int_{\{\Psi<-t_0\}}|\tilde{F}|^2e^{-\varphi}\tilde{c}(-\Psi)=
  \frac{G(T_1;c)}{\int_{T_1}^{+\infty}c(t)e^{-t}dt}
  \int_{t_0}^{+\infty}\tilde{c}(s)e^{-s}ds,
\end{split}
\end{equation*}
 where $\tilde{c}$ is a nonnegative measurable function on $(T,+\infty)$ and $T_1 \in (T,+\infty)$. Thus, if $\mathcal{H}^2(t;\tilde{c})\subset\mathcal{H}^2(t;c)$ for any $t>T$, then $G({h}^{-1}(r);\tilde c)$ is linear with respect to $r\in[0,\int_T^{+\infty}c(s)e^{-s}ds)$.
\end{Remark}

We recall a characterization of $G(t)=0$, where $t\ge T$.
\begin{Lemma}[\cite{GMY-boundary3}]\label{characterization of g(t)=0}Let $c\in\tilde{P}_{T,M,\Psi}$ satisfying that $\int_{T_1}^{+\infty}c(s)e^{-s}ds<+\infty$, where $T_1>T$. Let $t_0\ge T$.
The following two statements are equivalent:\\
$(1)$ $G(t_0)=0$;\\
$(2)$ $f_{z_0}\in
\mathcal{O} (K_M)_{z_0} \otimes J_{z_0}$, for any  $ z_0\in Z_0$.
\end{Lemma}

\subsection{Properties of $\mathcal{O}_{M,p}$-module $J_{p}$}
\label{sec:properties of module}
In this section, we present some properties of $\mathcal{O}_{M,p}$-module $J_{p}$. The notation used in this section can be referred to Section \ref{sec:Main result}.

Let $T\in [-\infty,+\infty)$. Denote $$\Psi:=\min\{\pi_1^*(\psi-2\log|F|),-T\}.$$
If $F(z)=0$ for some $z \in \Omega$, we set $\Psi(z,w)=-T$ for any $w\in Y$. Let $T_1>T$ be any real number.

Denote
$$\varphi_1:=2\max\{\pi_1^*(\psi)+T_1,\pi_1^*(2\log|F|)\},$$
and
$$\Psi_1:=\min\{\Psi,-T_1\}.$$
Denote $\varphi:=\pi_1^*(\varphi_{\Omega})+\pi_2^*(\varphi_Y)$. By definition we have $I(\Psi_1+\varphi)_p=I(\Psi+\varphi)_p$, for any $p\in M$. Recall that $\varphi_{\Omega}+\psi$ and $\psi$ are subharmonic functions on $\Omega$, hence for any relatively compact subset $U\Subset \Omega$, there exists a real number $a_1>0$ such that $e^{a_1\varphi_{\Omega}}$  is integrable on $U$. Note that $\varphi_Y$ is a plurisubharmonic function on $Y$ and then for any relatively compact set $W\Subset Y$, there exists a real number $a_2>0$ such that $e^{a_2\varphi_Y}$ is integrable on $W$.

Let $c(t)$ be a positive measurable function on $(T,+\infty)$ such that $c(t)\in\tilde{P}_{T,M,\Psi}$.
Let $dV_M$ be a continuous volume form on $M$. Let $p\in M$ be a point, denote that $H_p:=\{f_p\in J(\Psi)_p:\int_{\{\Psi<-t\}\cap V_0}|f|^2e^{-\varphi}c(-\Psi)dV_M<+\infty$  for some $t>T$  and  $V_0$  is an open neighborhood of $p\}$ and
$\mathcal{H}_p:=\{(F,p)\in \mathcal{O}_{M,p}:\int_{U_0}|F|^2e^{-\varphi-\varphi_1}c(-\Psi_1)dV_M<+\infty \text{ for some open neighborhood} \ U_0 \text{ of } p\}$.

As  $c(t)\in\tilde{P}_{T,M,\Psi}$, hence $c(t)e^{-t}$ is decreasing with respect to $t$ and we have $I(\Psi_1+\varphi)_p=I(\Psi+\varphi)_p\subset H_p$.
We also note that $\mathcal{H}_p$ is an ideal of $\mathcal{O}_{M,p}$.

In \cite{GMY-boundary3}, we proved the following proposition of $H_p/I(\varphi+\Psi_1)_p$.

\begin{Proposition}[\cite{GMY-boundary3}]
\label{module isomorphism}There exists an $\mathcal{O}_{M,p}$-module isomorphism $P:H_p/I(\varphi+\Psi_1)_p\to \mathcal{H}_p/\mathcal{I}(\varphi+\varphi_1+\Psi_1)_p$.
\end{Proposition}

We recall the following  closedness property of submodule of $\mathcal O_{\mathbb C^n,o}^q$.

\begin{Lemma}[see \cite{G-R}]
\label{closedness}
Let $N$ be a submodule of $\mathcal O_{\mathbb C^n,o}^q$, $1\leq q<+\infty$, let $f_j\in\mathcal O_{\mathbb C^n}(U)^q$ be a sequence of $q-$tuples holomorphic in an open neighborhood $U$ of the origin $o$. Assume that the $f_j$ converge uniformly in $U$ towards  a $q-$tuples $f\in\mathcal O_{\mathbb C^n}(U)^q$, assume furthermore that all germs $(f_{j},o)$ belong to $N$. Then $(f,o)\in N$.	
\end{Lemma}

We recall the following well-known result due to Skoda.
\begin{Lemma}[\cite{skoda1972}\label{l:skoda}]
	Let $u$ be a plurisubharmonic function on $\Delta^n\subset\mathbb{C}^n$. If $v(dd^cu,o )<1$, then $e^{-2u}$ is $L^1$ on a neighborhood of $o$, where $o\in\Delta^n$ is the origin.
\end{Lemma}

 Following from Lemma \ref{l:skoda} and Siu's Decomposition Theorem, we can obtain the following well-known result (see \cite{GMY-boundary3}).
\begin{Lemma}
	\label{l:1d-MIS}Let $\varphi$ be a subharmonic function on the unit disc $\Delta\subset\mathbb{C}$. Then we have $\mathcal{I}(\varphi)_o=(z^k)_o$ if and only if $v(dd^c(\varphi),o)\in[2k,2k+2)$.
\end{Lemma}

Recall that $\tilde{Z}_0\subset \Omega$ is a subset of $\cap_{t>T} \overline{\{(\pi_1)_*(\Psi)<-t\}}$.
Recall that $\tilde{Z}_1:=\{z\in \tilde{Z}_0:v(dd^c(\psi),z)\ge2ord_{z}(F)\}$, $\tilde{Z}_2:=\{z\in \tilde{Z}_0:v(dd^c(\psi),z)<2ord_{z}(F)\},$ and
 $\tilde{Z}_3:=\{z\in \tilde{Z}_0:v(dd^c(\psi),z)>2ord_{z}(F)\}$. We denote $Z_1:=\tilde{Z}_1\times Y$, $Z_2:=\tilde{Z}_2\times Y$ and $Z_3:=\tilde{Z}_3\times Y$ respectively.

\begin{Lemma}\label{lm:trivial module 1}Assume that $c(t)\ge 1$ near $+\infty$ and $v(dd^c(\psi),z_0)+v(dd^c(\psi+\varphi_{\Omega}),z_0)\notin \mathbb{Z}$ for any $z_0\in \tilde{Z}_2$. Then for any $p=(z,w)\in Z_2$, we have $H_{p}=I(\varphi+\Psi)_p$.
\end{Lemma}
\begin{proof}
	It follows from Proposition \ref{module isomorphism} that  $H_{p} =I(\varphi+\Psi_1)_{p}$ if and only if $\mathcal{H}_{p}=\mathcal{I}(\varphi+\Psi_1+\varphi_1)_{p}$, where $\mathcal{H}_{p}:=\{(h,p)\in\mathcal{O}_{M,p}:|h|^2e^{-\varphi-\varphi_1}c(-\Psi_1)$ is integrable near $p\}$. Now, we prove $\mathcal{H}_{p}=\mathcal{I}(\varphi+\Psi_1+\varphi_1)_{p}$.

 Without loss of generality, we can assume that $M=\Delta \times \Delta^{n-1}$, $z=0\in \Delta$ and $p=o\in M$ (the origin of $M$). As $c(t)e^{-t}$ is decreasing, we have $\mathcal{I}(\varphi+\Psi_1+\varphi_1)_{o}\subset \mathcal{H}_{o}$. For any $(h,o)\in \mathcal{H}_o$, there exists $r_1>0$ such that $\int_{\Delta_{r_1}\times \Delta_{r_1}^{n-1}}|h|^2e^{-\varphi-\varphi_1}c(-\Psi_1)<+\infty$, which implies that
\begin{equation}
	\label{eq:0316a}\int_{\Delta_{r_1}\times \Delta_{r_1}^{n-1}}|h|^2e^{-\varphi-\varphi_1}<+\infty,
\end{equation}
where $\Delta_{r_1}=\{z\in\mathbb{C}:|z|<r_1\}$. Hence we know that $h \in A^2(\Delta_{r_1} \times \Delta^{n-1}_{r_1},e^{-\varphi-\varphi_1})$. It follows from Lemma \ref{decomp integrable} and Remark \ref{remark of decomp integrable} that we know that $h(z,w)=\sum_{l\ge k}z^lF_l(w)$, where $\{F_l\}_{l\ge k}$ is a sequence of holomorphic functions on $W$ satisfying $F_{k}(w)\not\equiv 0$, $F_l\in A^2(\Delta^{n-1}_{r_1},e^{-\varphi_Y})$ and $|z|^{2l}e^{-\varphi_{\Omega}-(\pi_1)_*(\varphi_1)}$ is $L^1$ integrable near $0\in\Delta$ for any $l\ge k$.

Denote that $x_1:=v(dd^c(\psi),0)$ and $x_2:=v(dd^c(\varphi_{\Omega}+\psi),0)$. It follows from  Siu's Decomposition Theorem that
\begin{equation}
	\label{eq:0316b}\psi=x_1\log|z|+\tilde\psi,
\end{equation}
	where $\tilde\psi$ is a subharmonic function on $\Delta$ satisfying that $v(dd^c(\tilde\psi),0)=0$. As $v(dd^c(\psi),0)<2ord_0(F)$, we have
	\begin{equation}
		\label{eq:0316c}\psi\le\frac{1}{2}(\pi_1)_*(\varphi_1)=\max\{\psi+T_1,2\log|F|\}\le x_1\log|z|+C_1
	\end{equation}
	near $0$, where $C_1$ is a constant, which implies that $v(dd^c(\frac{1}{2}(\pi_1)_*(\varphi_1)),0)=x_1$. As $x_2=v(dd^c(\varphi_{\Omega}+\psi),0)$, we have
	\begin{equation}
		\label{eq:0316d}\varphi_{\Omega}+\psi\le x_2\log|z|+C_2
	\end{equation}
	near $0$, where $C_2$ is a constant.
	Combining inequality \eqref{eq:0316a}, equality \eqref{eq:0316b}, inequality \eqref{eq:0316c} and inequality \eqref{eq:0316d}, we get that there exists $r_2\in(0,r_1)$ such that
	\begin{equation}
		\label{eq:0316e}
		\begin{split}
			&\int_{\Delta_{r_2}}|z|^{2k-x_1-x_2}e^{\tilde\psi}\\
			\le &C_3\int_{\Delta_{r_2}}|z|^{2k}e^{-\frac{1}{2}(\pi_1)_*(\varphi_1)-\varphi_{\Omega}-\psi+\tilde\psi}\\
			=&C_3\int_{\Delta_{r_2}}|z|^{2k}e^{-\frac{1}{2}(\pi_1)_*(\varphi_1)-\varphi_{\Omega}-x_1\log|w|}\\
			\le&C_3e^{C_1}\int_{\Delta_{r_2}}|z|^{2k}e^{-\varphi_{\Omega}-(\pi_1)_*(\varphi_1)}\\
			<&+\infty.
		\end{split}
	\end{equation}
	For any $p>1$, as $v(dd^c(\tilde\psi),0)=0$, it follows from Lemma \ref{l:skoda} that there exists $r_3\in(0,r_2)$ such that $\int_{\Delta_{r_3}}e^{-\frac{q}{p}\tilde\psi}<+\infty$, where $\frac{1}{p}+\frac{1}{q}=1$.
	It follows from inequality \eqref{eq:0316e} and H\"older inequality that
	\begin{equation}
		\label{eq:0316f}
		\begin{split}
		&\int_{\Delta_{r_3}}|z|^{\frac{2k-x_1-x_2}{p}}\\
		\le&\left(\int_{\Delta_{r_3}}|z|^{2k-x_1-x_2}e^{\tilde\psi}\right)^{\frac{1}{p}}\left(\int_{\Delta_{r_3}}e^{-\frac{q}{p}\tilde\psi}\right)^{\frac{1}{q}}\\
		<&+\infty,
		\end{split}
	\end{equation}
	which shows that $|z|^{\frac{2k-x_1-x_2}{p}}$ is integrable near $o$ for any $p>1$.
	As $x_1+x_2=v(dd^c(\psi),0)+v(dd^c(\varphi_{\Omega}+\psi),0)\not\in \mathbb{Z}$,  we have $|z|^{2k-x_1-x_2}$ is integrable near $o$. Note that $v(dd^c(\varphi_{\Omega}+(\pi_1)_*(\Psi_1+\varphi_1)),0)=v(dd^c(\varphi_{\Omega}+\psi+\frac{1}{2}(\pi_1)_*(\varphi_1)),0)=x_1+x_2$. It follows from Lemma \ref{l:1d-MIS} that  $(z^k,0)\in\mathcal{I}(\varphi_\Omega+(\pi_1)_*(\Psi_1+\varphi_1))_0$.

It follows from $(z^l,0)\in\mathcal{I}(\varphi_\Omega+(\pi_1)_*(\Psi_1+\varphi_1))_0$ and $F_l\in A^2(\Delta^{n-1}_{r_1},e^{-\varphi_Y})$ for any $l\ge k$ that we have $z^lF_l(w)\in \mathcal{I}(\varphi+\Psi_1+\varphi_1)_o$ for any $l\ge k$. It follows from Lemma \ref{closedness} and  $h(z,w)=\sum_{l\ge k}z^lF_l(w)$ that we have $(h,o)\in \mathcal{I}(\varphi+\Psi_1+\varphi_1)_{o}$.
Hence we obtain that $\mathcal{H}_{o}=\mathcal{I}(\varphi+\Psi_1+\varphi_1)_{o}.$

Lemma \ref{lm:trivial module 1} is proved.
\end{proof}

\begin{Lemma}\label{lm:trivial module 2}For any $z\in \tilde{Z}_1$,	assume that one of the following two conditions holds:
	
	$(A)$ $\varphi+a\psi$ is  subharmonic near $z$ for some $a\in[0,1)$;
	
	$(B)$ $(\psi-2p_z\log|w|)(z)>-\infty$, where $p_z=\frac{1}{2}v(dd^c(\psi),z)$ and $w$ is a local coordinate on a neighborhood of $z$ satisfying that $w(z)=0$.

Let $c(t)\ge \frac{e^t}{t^2}$. Then for any $p=(z,w)\in Z_1\backslash Z_3$, we have $H_{p}=I(\varphi+\Psi)_p$.
\end{Lemma}
\begin{proof}It follows from Proposition \ref{module isomorphism} that  $H_{p} =I(\varphi+\Psi_1)_{p}$ if and only if $\mathcal{H}_{p}=\mathcal{I}(\varphi+\Psi_1+\varphi_1)_{p}$, where $\mathcal{H}_{p}:=\{(h,p)\in\mathcal{O}_{M,p}:|h|^2e^{-\varphi-\varphi_1}c(-\Psi_1)$ is integrable near $p\}$. Now, we prove $\mathcal{H}_{p}=\mathcal{I}(\varphi+\Psi_1+\varphi_1)_{p}$.

 Without loss of generality, we can assume that $M=\Delta \times \Delta^{n-1}$, $z=0\in \Delta$ and $p=o\in M$ (the origin of $M$). As $c(t)e^{-t}$ is decreasing, we have $\mathcal{I}(\varphi+\Psi_1+\varphi_1)_{o}\subset \mathcal{H}_{o}$. For any $(h,o)\in \mathcal{H}_o$, there exists $r_1>0$ such that $\int_{\Delta_{r_1}\times \Delta_{r_1}^{n-1}}|h|^2e^{-\varphi-\varphi_1}c(-\Psi_1)<+\infty$. Denote that $x_1:=v(dd^c(\psi),0)$ and $x_2:=v(dd^c(\varphi_{\Omega}+\psi),0)$.

	 Firstly, we prove $\mathcal{H}_{o}\subset \mathcal{I}(\varphi+\Psi_1+\varphi_1)_{o}$ under  condition $(A)$ ($\varphi+a\psi$ is subharmonic near $0$ for some $a\in[0,1)$).  As $c(t)\ge \frac{e^t}{t^2}$, we can assume that $c(t)\ge e^{at}$ when $t\ge T_1$. Then we have
\begin{equation}\nonumber
  \int_{\Delta_{r_1}\times \Delta_{r_1}^{n-1}}|h|^2e^{-\varphi-a\Psi_1-\varphi_1}
  \le \int_{\Delta_{r_1}\times \Delta_{r_1}^{n-1}}|h|^2e^{-\varphi-\varphi_1}c(-\Psi_1)<+\infty.
\end{equation}
Hence we know that $h\in A^2(M,e^{-\varphi-a\Psi_1-\varphi_1})$. It follows from Lemma \ref{decomp integrable} and Remark \ref{remark of decomp integrable} that we know that $h(z,w)=\sum_{l\ge k}z^lF_l(w)$, where $\{F_l\}_{l\ge k}$ is a sequence of holomorphic functions on $W$ satisfying $F_{k}(w)\not\equiv 0$, $F_l\in A^2(\Delta^{n-1},e^{-\varphi_Y})$ and $|z|^{2l}e^{-\varphi_{\Omega}-(\pi_1)_*(a\Psi_1+\varphi_1)}$ is $L^1$ integrable near $0\in\Delta$ for any $l\ge k$.

 Note that $\varphi_\Omega+(\pi_1)_*(a\Psi_1+\varphi_1)=\varphi_\Omega+a\psi+(2-a)\max\{\psi,2\log|F|\}$ and $v(dd^c(\varphi_\Omega+a\psi),0)=v(dd^c(\varphi_\Omega+\psi),0)-(1-a)v(dd^c(\psi),0)=x_2-(1-a)x_1$. As $v(dd^c(\psi),0)=2ord_{0}(F)$, we have $v(dd^c(\varphi_\Omega+(\pi_1)_*(a\Psi_1+\varphi_1)),0)=x_2-(1-a)x_1+(2-a)x_1=x_2+x_1$. Note that $v(dd^c(\varphi_\Omega+(\pi_1)_*(\Psi_1+\varphi_1)),0)=x_1+x_2$. It follows from $|z|^{2l}e^{-\varphi_{\Omega}-(\pi_1)_*(a\Psi_1+\varphi_1)}$ is $L^1$ integrable near $0\in\Delta$ for any $l\ge k$ and Lemma \ref{l:1d-MIS} that $(z^l,0)\in\mathcal{I}(\varphi_\Omega+(\pi_1)_*(\Psi_1+\varphi_1))_{0}$ for any $l\ge k$ .

It follows from $(z^l,0)\in\mathcal{I}(\varphi_\Omega+(\pi_1)_*(\Psi_1+\varphi_1))_0$ and $F_l\in A^2(\Delta^{n-1},e^{-\varphi_Y})$ for any $l\ge k$ that we have $z^lF_l(w)\in \mathcal{I}(\varphi+\Psi_1+\varphi_1)_o$ for any $l\ge k$. It follows from Lemma \ref{closedness} and  $h(z,w)=\sum_{l\ge k}z^lF_l(w)$ that we have $(h,o)\in \mathcal{I}(\varphi+\Psi_1+\varphi_1)_{o}$.
Hence we obtain that $\mathcal{H}_{o}=\mathcal{I}(\varphi+\Psi_1+\varphi_1)_{o}$ under  condition $(A)$.

Now, we prove $\mathcal{H}_{o}\subset \mathcal{I}(\varphi+\Psi_1+\varphi_1)_{o}$ under condition $(B)$ $\big((\psi-x_1\log|w|)(o)>-\infty\big)$. For any $(h,o)\in \mathcal{H}_o$, there exists $r_2>0$ such that $\int_{\Delta_{r_2}\times \Delta_{r_2}^{n-1}}|h|^2e^{-\varphi-\varphi_1}c(-\Psi_1)<+\infty$, which implies that
\begin{equation}
	\label{eq:0317b}\int_{\Delta_{r_2}\times \Delta_{r_2}^{n-1}}|h|^2e^{-\varphi-\varphi_1}<+\infty.
\end{equation}
Hence we know that $h \in A^2(\Delta_{r_2} \times \Delta^{n-1}_{r_2},e^{-\varphi-\varphi_1})$. It follows from Lemma \ref{decomp integrable} and Remark \ref{remark of decomp integrable} that we know that $h(z,w)=\sum_{l\ge k}z^lF_l(w)$, where $F_{k}(w)\not\equiv 0$, $\{F_l\}_{l\ge k}$ is a sequence of holomorphic functions on $W$ satisfying $F_l\in A^2(\Delta^{n-1}_{r_2},e^{-\varphi_Y})$ and $|z|^{2l}e^{-\varphi_{\Omega}-(\pi_1)_*(\varphi_1)}$ is $L^1$ integrable near $0\in\Delta$ for any $l\ge k$.

It follows from  Siu's Decomposition Theorem that
\begin{equation*}
\psi=x_1\log|z|+\tilde\psi,
\end{equation*}
where $\tilde\psi$ is a subharmonic function on $\Delta$ satisfying that $v(dd^c(\tilde\psi),0)=0$. Note that $\tilde{\psi}(0)>-\infty$ and $\varphi_\Omega+(\pi_1)_*(\varphi_1)\le \varphi_\Omega+2x_1\log|z|+C_1=\varphi_\Omega+\psi-\tilde\psi+x_1\log|z|\le (x_1+x_2)\log|z|-\tilde\psi+C_2$ near $0$, where $C_1$ and $C_2$ are constants. Note that $e^{\tilde\psi}$ is subharmonic, then it follows from $|z|^{2l}e^{-\varphi_{\Omega}-(\pi_1)_*(\varphi_1)}$ is $L^1$ integrable near $0\in\Delta$ for any $l\ge k$ and the  sub-mean value inequality that there exists $r_3\in(0,r_2)$ such that for any $l\ge k$, we have

\begin{equation}
			\label{eq:0317c}
			\begin{split}
			&\int_{\Delta_{r_3}}|z|^{2l-x_1-x_2}\\
			=&2\pi\int_0^{r_3}r^{2l+1-x_1-x_2}dr\\
			\le& \frac{1}{e^{\tilde\psi(o)}}\int_0^{r_3}\int_{0}^{2\pi}r^{2l+1-x_1-x_2}e^{\tilde\psi(re^{i\theta})}d\theta dr\\
			=&\frac{1}{e^{\tilde\psi(o)}}\int_{\Delta_{r_3}}|z|^{2l-x_1-x_2}e^{\tilde\psi}\\
			\le&\frac{C}{e^{\tilde\psi(o)}}\int_{\Delta_{r_3}}|z|^{2l}e^{-\varphi-(\pi_1)_*(\varphi_1)}\\
			<&+\infty.
			\end{split}
	\end{equation}

	As $v(dd^c(\varphi+(\pi_1)_*(\Psi_1+\varphi_1)),0)=x_1+x_2$, it follows from inequality \eqref{eq:0317c} and Lemma \ref{l:1d-MIS} that $(z^l,0)\in\mathcal{I}(\varphi+(\pi_1)_*(\Psi_1+\varphi_1))_{0}$ for any $l\ge k$.

It follows from $(z^l,0)\in\mathcal{I}(\varphi_\Omega+(\pi_1)_*(\Psi_1+\varphi_1))_0$ and $F_l\in A^2(\Delta^{n-1},e^{-\varphi_Y})$ for any $l\ge k$ that we have $z^lF_l(w)\in \mathcal{I}(\varphi+\Psi_1+\varphi_1)_o$ for any $l\ge k$. It follows from Lemma \ref{closedness} and  $h(z,w)=\sum_{l\ge k}z^lF_l(w)$ that we have $(h,o)\in \mathcal{I}(\varphi+\Psi_1+\varphi_1)_{o}$.
Hence we obtain that $\mathcal{H}_{o}=\mathcal{I}(\varphi+\Psi_1+\varphi_1)_{o}$ under  condition $(B)$.

Lemma \ref{lm:trivial module 2} is proved.
\end{proof}

\subsection{Some required lemmas}

In this section, we recall and present some required lemmas.
	\begin{Lemma}[see \cite{GY-concavity4}]
		\label{l:c(t)e^{-at}}
		Let $c(t)$ be a positive measurable function on $(0,+\infty)$, and let $a\in\mathbb{R}$. Assume that $\int_{t}^{+\infty}c(s)e^{-s}ds\in(0,+\infty)$ when $t$ near $+\infty$. Then we have
		
		$(1)$ $\lim_{t\rightarrow+\infty}\frac{\int_{t}^{+\infty}c(s)e^{-as}ds}{\int_t^{+\infty}c(s)e^{-s}ds}=1$ if and only if $a=1$;
		
		$(2)$ $\lim_{t\rightarrow+\infty}\frac{\int_{t}^{+\infty}c(s)e^{-as}ds}{\int_t^{+\infty}c(s)e^{-s}ds}=0$ if and only if $a>1$;
		
		$(3)$ $\lim_{t\rightarrow+\infty}\frac{\int_{t}^{+\infty}c(s)e^{-as}ds}{\int_t^{+\infty}c(s)e^{-s}ds}=+\infty$ if and only if $a<1$.
	\end{Lemma}
	
	The following Lemma belongs to Forn\ae ss
and Narasimhan on approximation property of plurisubharmonic
functions on Stein manifolds.

\begin{Lemma}[\cite{FN1980}]\label{l:FN1}
		Let $X$ be a Stein manifold and $\varphi \in PSH(X)$. Then there exists a sequence
		$\{\varphi_{n}\}_{n=1,\cdots}$ of smooth strongly plurisubharmonic functions such that
		$\varphi_{n} \downarrow \varphi$.
	\end{Lemma}

	Now, we recall some basic properties of the Green functions. Let $\Omega$ be an open Riemann surface, which admits a nontrivial Green function $G_{\Omega}$, and let $z_0\in\Omega$.
	
	\begin{Lemma}[see \cite{S-O69}, see also \cite{Tsuji}] 	\label{l:green-sup}
		Let $w$ be a local coordinate on a neighborhood of $z_0$ satisfying $w(z_0)=0$.  $G_{\Omega}(z,z_0)=\sup_{v\in\Delta_{\Omega}^*(z_0)}v(z)$, where $\Delta_{\Omega}^*(z_0)$ is the set of negative subharmonic function on $\Omega$ such that $v-\log|w|$ has a locally finite upper bound near $z_0$. Moreover, $G_{\Omega}(\cdot,z_0)$ is harmonic on $\Omega\backslash\{z_0\}$ and $G_{\Omega}(\cdot,z_0)-\log|w|$ is harmonic near $z_0$.
	\end{Lemma}

	\begin{Lemma}[see \cite{GY-concavity}]\label{l:G-compact}
		For any  open neighborhood $U$ of $z_0$, there exists $t>0$ such that $\{z\in\Omega: G_{\Omega}(z,z_0)<-t\}$ is a relatively compact subset of $U$.
	\end{Lemma}

 Let $S:=\{z_j:j\in \mathbb{Z}_{\ge1} \ \& \ j<\gamma\}$ be a discrete subset of the open Riemann surface $\Omega$, where $\gamma\in\mathbb{Z}_{\ge2}\cup\{+\infty\}$. We recall the following two lemmas about $\sum_{1\le j< \gamma}q_jG_{\Omega}(\cdot,z_j)$.

 	\begin{Lemma}[see \cite{GY-concavity3}]
		\label{l:green-sup2}
	 Let $\psi$ be a negative subharmonic function on $\Omega$ such that $\frac{1}{2}v(dd^c\psi,z_j)\geq q_j$ for any $j$, where $q_j>0$ is a constant. Then $2\sum_{1\le j< \gamma}q_jG_{\Omega}(\cdot,z_j)$ is a subharmonic function on $\Omega$ satisfying that $2\sum_{1\le j<\gamma }q_jG_{\Omega}(\cdot,z_j)\ge\psi$ and $2\sum_{1\le j<\gamma }q_jG_{\Omega}(\cdot,z_j)$ is harmonic on $\Omega\backslash S$.
	\end{Lemma}

The following lemma will be used in the proofs of Proposition \ref{p-finite-2} and Proposition \ref{p-infinite-2}.
	
	\begin{Lemma}[see \cite{GY-concavity3}]
		\label{l:psi=G}
		 Let $\psi$ be a negative plurisubharmonic function on $\Omega$ satisfying $\frac{1}{2}v(dd^c\psi,z_0)\geq q_j>0$ for any $j$, where $q_j$ is a constant. Assume that $\psi\not\equiv 2\sum_{1\le k<\gamma}q_jG_{\Omega}(\cdot,z_j)$. Let $l(t)$ be a positive Lebesgue measurable function on $(0,+\infty)$ satisfying $l$ is decreasing on $(0,+\infty)$ and $\int_0^{+\infty}l(t)dt<+\infty$. Then there exists a Lebesgue measurable subset $V$ of $M$ such that $l(-\psi(z))<l(-2\sum_{1\le k<\gamma}q_jG_{\Omega}(z,z_j))$ for any  $z\in V$ and $\mu(V)>0$, where $\mu$ is the Lebesgue measure on $\Omega$.
	\end{Lemma}

 Let $Y$ be an $n-1$ dimensional weakly pseudoconvex K\"{a}hler manifold. Let $M=\Omega\times Y$ be a complex manifold, and $K_M$ be the canonical line bundle on $M$. Let $\pi_1$, $\pi_2$ be the natural projections from $M$ to $\Omega$ and $Y$. Let $\psi_1$ be a subharmonic function on $\Omega$ such that $q_{z_j}=\frac{1}{2}v(dd^c\psi_1,z_j)>0$ for any $z_j\in S=\{z_j:1\le j<\gamma\}$, and let $\varphi_{\Omega}$ be a Lebesgue measurable function on $\Omega$ such that $\psi_1+\varphi_{\Omega}$ is subharmonic on $\Omega$. Let $\varphi_Y$ be a plurisubharmonic function on $Y$. Denote that $\psi:=\pi_1^*(\psi_1)$, $\varphi:=\pi_1^*(\varphi_{\Omega})+\pi_2^*(\varphi_Y)$. Using the Weierstrass Theorem on open Riemann surfaces (see \cite{OF81}) and
	 Siu's Decomposition Theorem, we have
	\[\varphi_{\Omega}+\psi_1=2\log |g_0|+2u_0,\]
	where $g_0$ is a holomorphic function on $\Omega$ and $u_0$ is a subharmonic function on $\Omega$ such that $v(dd^cu_0,z)\in [0,1)$ for any $z\in\Omega$.
	
	Let $w_j$ be a local coordinate on a neighborhood $V_{z_j}\subset\subset\Omega$ of $z_j$ satisfying $w_j(z_j)=0$ for $z_j\in \tilde Z_0$, where $V_{z_j}\cap V_{z_k}=\emptyset$ for any  $j\neq k$. Without loss of generality, assume that $w_j(V_{z_j})=\{w\in\mathbb{C}:|w|<s_j\}$, where $s_j>0$. Denote that $V_0:=\bigcup_{1\leq j<\gamma}V_{z_j}$. Assume that $g_0=d_jw_j^{k_j}h_j$ on $V_{z_j}$, where $d_j$ is a constant, $k_j$ is a nonnegative integer, and $h_j$ is a holomorphic function on $V_{z_j}$ such that $h_j(z_j)=1$ for any $j$ ($1\leq j<\gamma$).
	
	Let $c(t)$ be a positive measurable function on $(0,+\infty)$ satisfying that $c(t)e^{-t}$ is decreasing and $\int_0^{+\infty}c(t)e^{-t}dt<+\infty$. We recall the following optimal $L^2$ extension theorem.

	\begin{Theorem}[see \cite{BGY-concavity5}]\label{L2ext-finite-f}
		Let $F$ be a holomorphic $(n,0)$ form on $V_0\times Y$ such that $F=\pi_1^*(w_j^{\tilde{k}_j}f_jdw_j)\wedge \pi_2^*(\tilde{F}_j)$ on $V_{z_j}\times Y$, where $\tilde{k}_j$ is a nonnegative integer, $f_j$ is a holomorphic function on $V_{z_j}$ such that $f_j(z_j)=a_j\in\mathbb{C}\setminus\{0\}$, and $\tilde{F}_j$ is a holomorphic $(n-1,0)$ form on $Y$ for any $1\leq j<\gamma$.
		
		Denote that $I_F:=\{j : 1\leq j<\gamma\,\& \,\tilde{k}_j+1-k_j\leq 0\}$. Assume that $\tilde{k}_j+1-k_j=0$ and $u_0(z_j)>-\infty$ for $j\in I_F$. If
		\begin{equation}
			\int_Y|\tilde{F}_j|^2e^{-\varphi_Y}<+\infty
		\end{equation}
		for any  $1\leq j<\gamma$, and
		\begin{equation}
			\sum_{j\in I_F}\frac{2\pi|a_j|^2 e^{-2u_0(z_j)}}{q_{z_j}|d_j|^2}\int_Y|\tilde{F}_j|^2e^{-\varphi_Y}<+\infty,
		\end{equation}
		then there exists a holomorphic $(n,0)$ form $\tilde{F}$ on $M$, such that $(\tilde{F}-F,(z_j,y))\in(\mathcal{O}(K_M)\otimes\mathcal{I}(\varphi+\psi))_{(z_j,y)}$ for any $1\leq j<\gamma$ and $y\in Y$, and
		\begin{equation}
			\int_M|\tilde{F}|^2e^{-\varphi}c(-\psi)\leq \left(\int_0^{+\infty}c(s)e^{-s}ds\right)\sum_{j\in I_F}\frac{2\pi|a_j|^2 e^{-2u_0(z_j)}}{q_{z_j}|d_j|^2}\int_Y|\tilde{F}_j|^2e^{-\varphi_Y}.	
		\end{equation}
	\end{Theorem}

	The following lemma will be used in the proofs of Proposition \ref{p-finite-2} and Proposition \ref{p-infinite-2}.
	\begin{Lemma}\label{l:inte}
		For any $1\le j<\gamma$, assume that one of the following conditions holds:
	
	$(A)$ $\varphi_{\Omega}+a\psi_1$ is subharmonic near $z_j$ for some $a\in[0,1)$;
	
	$(B)$ $(\psi_1-2q_{z_j}\log|w_j|)(z_j)>-\infty$, where $q_{z_j}=\frac{1}{2}v(dd^c(\psi_1),z_j)$.
	
	Let $F$ be a holomorphic $(n,0)$ form on $M$ such that $F=\sum_{l\ge\tilde{k}_j}\pi_1^*(w_j^ldw_j)\wedge\pi_2^*(F_{j,l})$ on $V_{z_j}\times Y$ for any $j$  according to Lemma \ref{decomp}, where $\tilde{k}_j$ is a nonnegative integer, $F_{j,l}$ is a holomorphic $(n-1,0)$ form on $Y$ satisfying that $\tilde{F}_j:=F_{j,\tilde{k}_j}\not\equiv 0$. Denote that
		\begin{equation*}
			I_F:=\{j:1\leq j<\gamma \ \& \ \tilde{k}_j+1-k_j\leq 0\}.
		\end{equation*}
		Assume that there exists a constant $C>0$ such that
		\begin{equation*}
	\int_{\{\psi<-t\}}|F|^2e^{-\varphi}\tilde c(-\psi)\le C\int_t^{+\infty}\tilde c(s)e^{-s}ds		
	\end{equation*}
		for any $t\ge0$ and any nonnegative Lebesgue measurable function $\tilde c(t)$ on $(0,+\infty)$.
		 Then $\tilde{k}_j+1-k_j=0$ for any $j\in I_F$, $\int_Y| F_{j,l}|^2e^{-\varphi_Y}<+\infty$ for any $j,l$, and
		\begin{equation*}
			\sum_{j\in I_F}\frac{2\pi e^{-2u_0(z_j)}}{q_{z_j}|d_j|^2}\int_Y|\tilde{F}_j|^2e^{-\varphi_Y}\le C.
		\end{equation*}
	\end{Lemma}

\begin{proof}
	Firstly, we prove that $\int_Y| F_{j,l}|^2e^{-\varphi_Y}<+\infty$ for any $j,l$.	 Fixed $z_j$,  there exists $r_j>0$ such that $\{|w_j|<r_j\}\subset V_{z_j}$ and
	\begin{equation}
		\label{eq:0405a}\int_{\{|w_j|<r_j\}\times Y}|F|^2e^{-\varphi}\tilde c(-\psi)<+\infty
	\end{equation}
	for any nonnegative Lebesgue measurable function $\tilde c(t)$ on $(0,+\infty)$ satisfying $\int_0^{+\infty}\tilde c(s)e^{-s}ds<+\infty$.
	
	 If there exists $a\in[0,1)$ such that $\varphi_{\Omega}+a\psi_1$ near $z_j$, note that there exists a positive Lebesgue measurable function $\tilde c(t)$ on $(0,+\infty)$ such that $\tilde c(t)e^{-at}$ (denoted by $b(t)$) is increasing near $+\infty$ and $\int_0^{+\infty}\tilde c(s)e^{-s}ds<+\infty$, then there exist $\tilde r_j\in(0,r_j)$ and  $C_1<0$ such that $b(t)$ is increasing on $(-C_1,+\infty)$
	 $$\varphi_{\Omega}+a\psi_1\le C_1,$$
	 $$\psi_1\le C_1$$
	 on $\{|w_j|<\tilde r_j\}$.
 Inequality \eqref{eq:0405a} shows that
	 \begin{equation}
	 	\label{eq:0405b}\begin{split}
	 	&\int_{\{|w_j|<\tilde r_j\}\times Y}|F|^2e^{-\pi_2^*(\varphi_Y)-C_1}b(-C_1)\\
	 	\le&\int_{\{|w_j|<\tilde r_j\}\times Y}|F|^2e^{-\pi_1^*(\varphi_{\Omega}+a\psi_1)-\pi_2^*(\varphi_Y)}b(-\pi_1^*(\psi_1))\\
	 	<&+\infty.
	 	\end{split}
	 \end{equation}	
It follows from Lemma 	\ref{decomp} and inequality \eqref{eq:0405b} that $\int_Y|F_{j,l}|^2e^{-\varphi_Y}<+\infty$ for any $l\ge\tilde k_j$.
	
	If $(\psi_1-2q_{z_j}\log|w_j|)(z_j)>-\infty$, then $\tilde\psi_1=\psi_1-2q_{z_j}\log|w_j|$ is a subharmonic function on $V_{z_j}$ satisfying $\tilde\psi_1(z_j)>-\infty$. Choosing $\tilde c\equiv1$, it  follows from Fubini's Theorem and sub-mean value inequality of subharmonic functions  that
	\begin{equation}
		\label{eq:0405c}\begin{split}
					&\int_{\{|w_j|<r_j\}\times Y}|F|^2e^{-\varphi}\\
					\ge&C_2\int_{\{|w_j|<r_j\}\times Y}|F|^2e^{\pi_1^*(\tilde\psi_1+2q_{z_j}\log|w_j|)-\pi_2^*(\varphi_Y)}\\
					\ge&4\pi e^{\tilde\psi_1(z_j)} C_2\int_0^{r_j}|r|^{2\tilde k_j+2q_{z_j}+1}dr\int_Y|\tilde F_j|^2e^{-\varphi_Y},
					\end{split}
	\end{equation}
	where $C_2>0$ is a constant.
	Combining inequality \eqref{eq:0405a} and inequality \eqref{eq:0405c}, we get that $\int_Y|\tilde F_j|^2e^{-\varphi_Y}<+\infty$. It follows from Lemma \ref{decomp integrable} and Remark \ref{remark of decomp integrable} that $|w_j|^{2l}e^{-\varphi_{\Omega}}$ is integrable near $z_j$ for any $l\ge\tilde k_j$, hence there exists $r'_j\in(0,r_j)$ such that
	\begin{equation}\label{eq:220406a}
		\begin{split}
			&\int_{\{|w_j|<r'_j\}\times Y}|\sum_{l\ge\tilde k_j+1}\pi_1^*(w_j^ldw_j)\wedge\pi_2^*(F_{j,l})|^2e^{-\varphi}\\
			\le&2\int_{\{|w_j|<r'_j\}\times Y}|\pi_1^*(w_j^{\tilde k_j}dw_j)\wedge\pi_2^*(\tilde F_j)|^2e^{-\varphi}+2\int_{\{|w_j|<r'_j\}\times Y}|F|^2e^{-\varphi}\\
			<&+\infty.
		\end{split}
	\end{equation}
Combining 	Fubini's Theorem and sub-mean value inequality of subharmonic functions, inequality \eqref{eq:220406a} shows that $\int_Y|F_{j,\tilde k_j+1}|^2e^{-\varphi_Y}<+\infty$. Hence, we know that $\int_Y|F_{j,l}|^2e^{-\varphi_Y}<+\infty$ for any $l\ge\tilde k_j$ by induction.

Now, we prove that 	$\tilde{k}_j+1-k_j=0$ for any $j\in I_F$ and
		\begin{equation*}
			\sum_{j\in I_F}\frac{2\pi e^{-2u_0(z_j)}}{q_{z_j}|d_j|^2}\int_Y|\tilde{F}_j|^2e^{-\varphi_Y}\le C.
		\end{equation*}
		According to  Siu's Decomposition Theorem, Lemma \ref{l:green-sup} and \ref{l:green-sup2}, we can assume that
		\begin{equation}
			\psi_1=2\sum_{1\leq j<\gamma}q_{z_j}G_{\Omega}(\cdot,z_j)+\psi_0,
		\end{equation}
		where $\psi_0$ is a negative subharmonic function on $\Omega$ such that $v(dd^c(\psi_0),z_j)=0$ for any  $1\leq j<\gamma$.
Lemma \ref{l:FN1} shows that there exist smooth subharmonic functions $\{u_l\}_{l\in\mathbb{Z}_{\ge1}}$ and $\{\tilde\psi_l\}_{l\in\mathbb{Z}_{\ge1}}$ on $\Omega$ such that $u_l$ and $\tilde\psi_l$ are decreasingly convergent to $u_0$ and $\psi_0$ with respect to $l$, respectively. For any positive integer $m$ satisfying $2\le m\le\gamma$, denote that $I_m:=\{j\in I_F:1\le j<m\}$. Denote that
$$G:=2\sum_{1\le j<\gamma}q_{z_j}G_{\Omega}(\cdot,z_j).$$

Fixing a positive integer $m$, there exists $a\in[0,1)$ such that $\varphi_{\Omega}+a\psi_1$ is subharmonic near $z_j$ for any $1\le j<m$ satisfying  $(\psi_1-2q_{z_j}\log|w_j|)(z_j)=-\infty$. Note that there exists a positive Lebesgue measurable function $\tilde c(t)$ on $(0,+\infty)$ such that $\tilde c(t)e^{-at}$ (denoted by $b(t)$) is increasing near $+\infty$ and $\int_0^{+\infty}\tilde c(s)e^{-s}ds<+\infty$.  Note that
\begin{equation}
	\label{eq:0405d}\begin{split}
		\int_{\{\psi<-t\}}|F|^2e^{-\varphi}\tilde c(-\psi)&\ge\sum_{1\le j<m}\int_{\{\psi<-t\}\cap (V_{z_j}\times Y)}|F|^2e^{-\varphi}\tilde c(-\psi)\\
		&\ge\sum_{j\in I_m}\int_{\{\psi<-t\}\cap (V_{z_j}\times Y)}|F|^2e^{-\varphi}\tilde c(-\psi).
	\end{split}
\end{equation}
Fixing $j\in I_m$, we will prove that $\tilde k_j+1=k_j$ and
\begin{equation}
	\label{eq:0405e}\liminf_{t\rightarrow+\infty}\frac{\int_{\{\psi<-t\}\cap (V_{z_j}\times Y)}|F|^2e^{-\varphi}\tilde c(-\psi)}{\int_t^{+\infty}\tilde c(s)e^{-s}ds}\geq \frac{2\pi e^{-2u_0(z_j)}}{q_{z_j}|d_j|^2}\int_Y|\tilde{F}_j|^2e^{-\varphi_Y}.
\end{equation}

If $(\psi_1-2q_{z_j}\log|w_j|)(z_j)>-\infty$, we have $\psi_0(z_j)>-\infty$. For any $\epsilon>0$, there exists $r_{j,1}>0$ such that
 $U_j:=\{|w_j(z)|<r_{j,1}:z\in V_{z_j} \}\Subset V_{z_j}$,
 $$\sup_{z\in U_j}2|u_l(z)-u_l(z_j)|<\epsilon,$$
 and
$$\sup_{z\in U_j}|H_j(z)-H_j(z_j)|<\epsilon,$$ where $H_j=G-2q_{z_j}\log|w_j|+\tilde\psi_l+\epsilon$ is a smooth function on $V_{z_j}$. There exists $t_1>0$ such that $\{2q_{z_j}\log|w_j|+H_j(z_j)<-t_1\}\Subset U_j$ and $\tilde c$ is increasing on $[t_1,+\infty)$, then
we get that
\begin{equation}
	\label{eq:0406a}\begin{split}
	&\int_{\{\psi<-t\}\cap (V_{z_j}\times Y)}|F|^2e^{-\varphi}\tilde c(-\psi)\\
	\geq
		&\int_{\{\pi_1^*(G+\tilde\psi_l)<-t\}\cap (V_{z_j}\times Y)}|F|^2e^{-\varphi}\tilde c(-\psi)\\
		\ge&\int_{\{2q_{z_j}\log|w_j|+H_j(z_j)<-t\}\times Y}|F|^2e^{\pi_1^*(-2\log|g_0|-2u_l(z_j)-\epsilon+G+\psi_0)+\pi_2^*(-\varphi_Y)}\\
		&\times \tilde c(\pi_1^*(-2q_{z_j}\log|w_j|-H_j(z_j)))
	\end{split}
\end{equation}
for $t\ge t_1$.
Note that there exists a holomorphic function $\tilde g_j$ on $U_j$ such that $|\tilde g_j|^2=e^{\frac{G}{q_{z_j}}}$. Note that $ord_{z_j}(\tilde g_j)=1$, and let $\tilde d_j=\lim_{z\rightarrow z_j}\frac{\tilde g_j(z)}{w_j(z)}\in\mathbb{C}\backslash\{0\}$. Note that $\lim_{z\rightarrow z_j}\frac{ g_0(z)}{w^{k_j}_j(z)}=d_j\in\mathbb{C}\backslash\{0\}$. Without loss of generality, assume that $\{g_0(z)=0:z\in U_j\}=\{z_j\}$. It follows from Fubini's Theorem, sub-mean value inequality of subharmonic functions and inequality \eqref{eq:0406a} that
\begin{equation}
	\label{eq:0406b}\begin{split}
		&\int_{\{\psi<-t\}\cap (V_{z_j}\times Y)}|F|^2e^{-\varphi}\tilde c(-\psi)\\
		\ge&\int_{\{2q_{z_j}\log|w_j|+H_j(z_j)<-t\}\times Y}|F|^2|\pi_1^*(\tilde g_j)|^{2q_{z_j}}e^{\pi_1^*(-2\log|g_0|-2u_l(z_j)-\epsilon+\psi_0)+\pi_2^*(-\varphi_Y)}\\
		&\times \tilde c(\pi_1^*(-2q_{z_j}\log|w_j|-H_j(z_j)))\\
		\ge&4\pi\frac{|\tilde d_j|^{2q_{z_j}}e^{-2u_l(z_j)-\epsilon+\psi_0(z_j)}}{|d_j|^2}
		\int_0^{e^{-\frac{t+H_j(z_j)}{2q_{z_j}}}}r^{2\tilde k_j+2q_{z_j}-2k_j+1}\tilde c(-2q_{z_j}\log|r|-H_j(z_j))dr\\
		&\times\int_Y|\tilde F_j|^2e^{-\varphi_Y}\\
		=&\frac{2\pi|\tilde d_j|^{2q_{z_j}}e^{-2u_l(z_j)-\epsilon+\psi_0(z_j)}}{q_{z_j}|d_j|^2}e^{-\left(\frac{\tilde k_j-k_j+1}{q_{z_j}}+1\right)H_j(z_j)}\int_t^{+\infty}e^{-\left(\frac{\tilde k_j-k_j+1}{q_{z_j}}+1\right)s}\tilde c(s)ds\\
		&\times\int_Y|\tilde F_j|^2e^{-\varphi_Y}
			\end{split}
\end{equation}
for $t\ge t_1$.
As $\liminf_{t\rightarrow+\infty}\frac{\int_{\{\psi<-t\}\cap (V_{z_j}\times Y)}|F|^2e^{-\varphi}\tilde c(-\psi)}{\int_t^{+\infty}\tilde c(s)e^{-s}ds}<+\infty$ and $\tilde F_j\not\equiv0$, it follows from Lemma \ref{l:c(t)e^{-at}} and $\tilde k_j-k_j+1\le0$ that $\tilde k_j-k_j+1=0$. Hence letting $\epsilon\rightarrow0+0$, inequality \eqref{eq:0406b} implies that
\begin{equation}
	\label{eq:0406c:correct}\begin{split}
	&\liminf_{t\rightarrow+\infty}\frac{\int_{\{\psi<-t\}\cap (V_{z_j}\times Y)}|F|^2e^{-\varphi}\tilde c(-\psi)}{\int_t^{+\infty}\tilde c(s)e^{-s}ds}\\
	\ge&\frac{2\pi|\tilde d_j|^{2q_{z_j}}e^{-2u_l(z_j)+\psi_0(z_j)-(G-2q_{z_j}\log|w_j|)(z_j)-\tilde\psi_l(z_j)}}{q_{z_j}|d_j|^2}
		\times\int_Y|\tilde F_j|^2e^{-\varphi_Y}.
	\end{split}
\end{equation}
Note that $|\tilde d_j|^{2q_{z_j}}e^{-(G-2q_{z_j}\log|w_j|)(z_j)}=1$, $\lim_{l\rightarrow+\infty}u_j(z_j)=u_0(z_j)$ and $\lim_{l\rightarrow+\infty}\tilde\psi_l(z_j)=\psi_0(z_j)$. Letting $l\rightarrow+\infty$, it follows from  inequality \eqref{eq:0406c:correct} that
\begin{equation*}
	\liminf_{t\rightarrow+\infty}\frac{\int_{\{\psi<-t\}\cap (V_{z_j}\times Y)}|F|^2e^{-\varphi}\tilde c(-\psi)}{\int_t^{+\infty}\tilde c(s)e^{-s}ds}\geq \frac{2\pi e^{-2u_0(z_j)}}{q_{z_j}|d_j|^2}\int_Y|\tilde{F}_j|^2e^{-\varphi_Y}.
\end{equation*}

If $(\psi_1-2q_z\log|w_j|)(z_j)=-\infty$, then $\varphi_{\Omega}+a\psi_1$ is subharmonic on a neighborhood $\tilde V_{z_j}\subset V_{z_j}$ of $z_j$. Denote that
$$\varphi_0:=\varphi_{\Omega}+a\psi_1-2s_1\log|w_j|$$
is a subharmonic function on $\tilde V_{z_j}$, where $s_1=\frac{1}{2}v(dd^c(\varphi_{\Omega}+a\psi_1),z_j)$. Following from Lemma \ref{l:FN1}, there exist smooth subharmonic functions $\{\varphi_l\}_{l\in\mathbb{Z}_{\ge1}}$ and $\{\tilde\psi_l\}_{l\in\mathbb{Z}_{\ge1}}$ on $\Omega$ such that $\varphi_l$ and $\tilde\psi_l$ are decreasingly convergent to $\varphi_0$ and $\psi_0$ with respect to $l$, respectively.
For any $\epsilon>0$, there exists $r_{j,2}>0$ such that  $U'_j:=\{|w_j(z)|<r_{j,2}:z\in V_{z_j}\}\Subset \tilde V_{z_j}$,
$$\sup_{z\in U'_j}|\varphi_l(z)-\varphi_l(z_j)|<\epsilon,$$ and
$$\sup_{z\in U_j}|H_j(z)-H_j(z_j)|<\epsilon,$$ where $H_j=G-2q_{z_j}\log|w_j|+\tilde\psi_l+\epsilon$ is a smooth function on $\tilde V_{z_j}$.
There exists $t_2>0$ such that $\{2q_{z_j}\log|w_j|+H_j(z_j)<-t_2\}\Subset U'_j$ and $b(t)=\tilde c(t)e^{-at}$ is increasing on $[t_2,+\infty)$, then
we get that
\begin{equation}
	\label{eq:220406b}\begin{split}
	&\int_{\{\psi<-t\}\cap (V_{z_j}\times Y)}|F|^2e^{-\varphi}\tilde c(-\psi)\\
	\geq
		&\int_{\{\pi_1^*(G+\tilde\psi_l)<-t\}\cap (V_{z_j}\times Y)}|F|^2e^{-\pi_1^*(\varphi_{\Omega}+a\psi_1)-\pi_2^*(\varphi_Y)}b(-\pi_1^*(\psi_1))\\
		\ge&\int_{\{2q_{z_j}\log|w_j|+H_j(z_j)<-t\}\times Y}|F|^2e^{\pi_1^*(-2s_1\log|w_j|-\varphi_l(z_j)-\epsilon)+\pi_2^*(-\varphi_Y)}\\
		&\times b(\pi_1^*(-2q_{z_j}\log|w_j|-H_j(z_j)))
	\end{split}
\end{equation}
for $t\ge t_2$.
 It follows from Fubini's Theorem, sub-mean value inequality of subharmonic functions and inequality \eqref{eq:220406b} that
\begin{equation}
	\label{eq:220406c}\begin{split}
		&\int_{\{\psi<-t\}\cap (V_{z_j}\times Y)}|F|^2e^{-\varphi}\tilde c(-\psi)\\
		\ge&4\pi e^{-\varphi_l(z_j)-\epsilon}
		\int_0^{e^{-\frac{t+H_j(z_j)}{2q_{z_j}}}}r^{2\tilde k_j-2s_1+1}b(-2q_{z_j}\log|r|-H_j(z_j))dr\times\int_Y|\tilde F_j|^2e^{-\varphi_Y}\\
		=&\frac{2\pi e^{-\varphi_l(z_j)-\epsilon}}{q_{z_j}}e^{-\left(\frac{\tilde k_j-s_1+1}{q_{z_j}}\right)H_j(z_j)}\int_t^{+\infty}e^{-\left(\frac{\tilde k_j-s_1+1}{q_{z_j}}\right)s}b(s)ds\times\int_Y|\tilde F_j|^2e^{-\varphi_Y}\\
		=&\frac{2\pi e^{-\varphi_l(z_j)-\epsilon}}{q_{z_j}}e^{-\left(\frac{\tilde k_j-s_1+1}{q_{z_j}}\right)H_j(z_j)}\int_t^{+\infty}e^{-\left(\frac{\tilde k_j-s_1+1}{q_{z_j}}+a\right)s}\tilde c(s)ds\times\int_Y|\tilde F_j|^2e^{-\varphi_Y}.
			\end{split}
\end{equation}
for $t\ge t_1$. Note that
\begin{displaymath}
	\begin{split}
	s_1+(1-a)q_{z_j}&=\frac{1}{2}(v(dd^c(\varphi_{\Omega}+a\psi_1),z_j)+(1-a)v(dd^c(\psi_1),z_j))	\\
	&=\frac{1}{2}(v(dd^c(\varphi_{\Omega}+\psi_1),z_j))	\\
	&=ord_{z_j}(g_0)\\
	&=k_j,
	\end{split}
\end{displaymath}
which implies that
$$\frac{\tilde k_j-s_1+1}{q_{z_j}}+a=\frac{\tilde k_j-k_j+1}{q_{z_j}}+1.$$
As $\liminf_{t\rightarrow+\infty}\frac{\int_{\{\psi<-t\}\cap (V_{z_j}\times Y)}|F|^2e^{-\varphi}\tilde c(-\psi)}{\int_t^{+\infty}\tilde c(s)e^{-s}ds}<+\infty$ and $\tilde F_j\not\equiv0$, it follows from Lemma \ref{l:c(t)e^{-at}} and $\tilde k_j-k_j+1\le0$ that $\tilde k_j-k_j+1=0$. Hence letting $\epsilon\rightarrow0+0$, inequality \eqref{eq:220406c} implies that
\begin{equation}
	\label{eq:0406c}\begin{split}
	&\liminf_{t\rightarrow+\infty}\frac{\int_{\{\psi<-t\}\cap (V_{z_j}\times Y)}|F|^2e^{-\varphi}\tilde c(-\psi)}{\int_t^{+\infty}\tilde c(s)e^{-s}ds}\\
	\ge&\frac{2\pi e^{-\varphi_l(z_j)-(1-a)(G+\psi_l-2q_{z_j}\log|w_j|)(z_j)}}{q_{z_j}}\int_Y|\tilde F_j|^2e^{-\varphi_Y}.
	\end{split}
\end{equation}
Note that \begin{displaymath}
	\begin{split}
		&\lim_{l\rightarrow+\infty}\varphi_l(z_j)+(1-a)(G+\psi_l-2q_{z_j}\log|w_j|)(z_j)\\
		=&(1-a)(G+\psi_0-2q_{z_j}\log|w_j|)(z_j)+\varphi_0(z_j)\\
		=&(1-a)(\psi_1-2q_{z_j}\log|w_j|)(z_j)+(\varphi_{\Omega}+a\psi_1-2s_1\log|w_j|)(z_j)\\
		=&(\varphi_{\Omega}+\psi_1-2k_j\log|w_j|)(z_j)\\
		=&2u_0(z_j)+2\log|d_j|,
	\end{split}
\end{displaymath}
then it follows from inequality \eqref{eq:0406c} that
\begin{equation*}
	\liminf_{t\rightarrow+\infty}\frac{\int_{\{\psi<-t\}\cap (V_{z_j}\times Y)}|F|^2e^{-\varphi}\tilde c(-\psi)}{\int_t^{+\infty}\tilde c(s)e^{-s}ds}\geq \frac{2\pi e^{-2u_0(z_j)}}{q_{z_j}|d_j|^2}\int_Y|\tilde{F}_j|^2e^{-\varphi_Y}.
\end{equation*}

Hence for any $m$ satisfying $2\le m<\gamma$, we have $\tilde k_j-k_j+1=0$ for any $j\in I_m$, and inequality \eqref{eq:0405d} implies that $$\sum_{j\in I_m}\frac{2\pi e^{-2u_0(z_j)}}{q_{z_j}|d_j|^2}\int_Y|\tilde{F}_j|^2e^{-\varphi_Y}\le \liminf_{t\rightarrow+\infty}\frac{\int_{\{\psi<-t\}}|F|^2e^{-\varphi}\tilde c(-\psi)}{\int_t^{+\infty}\tilde c(s)e^{-s}ds}\le C.$$
By the arbitrariness of $m$, we have  $\tilde{k}_j+1-k_j=0$ for any $j\in I_F$ and
		\begin{equation*}
			\sum_{j\in I_F}\frac{2\pi e^{-2u_0(z_j)}}{q_{z_j}|d_j|^2}\int_Y|\tilde{F}_j|^2e^{-\varphi_Y}\le C.
		\end{equation*}
		\end{proof}

	\subsection{Linearity on fibrations over open Riemann surfaces at inner points}In this section, we recall and present some results about the concavity  degenerating to linearity on fibrations over open Riemann surfaces at inner points.

Let $\Omega$ be an open Riemann surface, which admits a nontrival Green function $G_{\Omega}$, and let $K_{\Omega}$ be the canonical (holomorphic) line bundle on $\Omega$. Let $Y$ be an $n-1$ dimensional weakly pseudoconvex K\"{a}hler manifold. Let $M=\Omega\times Y$ be a complex manifold, and $K_M$ be the canonical line bundle on $M$. Let $\pi_1$, $\pi_2$ be the natural projections from $M$ to $\Omega$ and $Y$.
	
	Let $\tilde Z_0$ be a (closed) analytic subset of $\Omega$ and denote $Z_0:=\pi_1^{-1}(\tilde Z_0)$ be an analytic subset of $M$. Let $\psi_1$ be a negative subharmonic function on $\Omega$ such that $\psi_1(z)=-\infty$ for any $z\in \tilde Z_0$, and let $\varphi_{\Omega}$ be a Lebesgue measurable function on $\Omega$ such that $\varphi_{\Omega}+\psi_1$ is subharmonic on $\Omega$. Let $\varphi_{Y}$ be a plurisubharmonic function on $Y$. Denote that $\psi:=\pi_1^*(\psi_1)$, $\varphi=\pi_1^*(\varphi_{\Omega})+\pi_2^*(\varphi_{Y})$. Let $c$ be a positive function on $(0,+\infty)$ sucht that $\int_0^{+\infty}c(t)e^{-t}dt<+\infty$, $c(t)e^{-t}$ is decreasing on $(0,+\infty)$ and $e^{-\varphi}c(-\psi)$ has a positive lower bound on any compact subset of $M\setminus \pi_1^{-1}(E)$, where $E\subset\{\psi_1=-\infty\}$ is a discrete subset of $\Omega$.
	
	 Let $f$ be a holomorphic $(n,0)$ form on a neighborhood of $Z_0$. Denote
	\begin{flalign*}
		\begin{split}
			\inf\Bigg\{\int_{\{\psi<-t\}}|\tilde{f}|^2e^{-\varphi}&c(-\psi) :\tilde{f}\in H^0(\{\psi<-t\},\mathcal{O}(K_M))  \\
			& \& \,(\tilde{f}-f,p)\in \mathcal{O}(K_M)_{p}\otimes\mathcal{I}(\varphi+\psi)_{p}\text{ for  any } p\in Z_0  \Bigg\}
		\end{split}
	\end{flalign*}
	by $G(t)$ for any $t\in [0,+\infty)$.

Recall some notations related to open Riemann surfaces (see \cite{OF81}, see also \cite{guan-zhou13ap,GY-concavity,GMY-concavity2}). Let $P:\Delta\rightarrow \Omega$ be the universal covering from the unite disc $\Delta$ to $\Omega$. The holomorphic function $\hat{f}$  on $\Delta$ is called a multiplicative function, if there is a character $\chi$, which is the representation of the fundamental group of $\Omega$ such that $g^*(\hat{f})=\chi(g)\hat{f}$, where $|\chi|=1$ and $g$ is an element of the fundamental group of $\Omega$. Denote the set of such $\hat{f}$ by $\mathcal{O}^{\chi}(\Omega)$ .
	
	It is known that for any harmonic function $u$ of $\Omega$, there exists a $\chi_u$ and a multiplicative function $f_u\in\mathcal{O}^{\chi_u}(\Omega)$ such that $|f_u|=P^*(e^u)$. If $u_1-u_2=\log|\hat{f}|$, then $\chi_{u_1}=\chi_{u_2}$, where $u_1$ and $u_2$ are harmonic functions on $\Omega$ and $\hat{f}$ is a holomorphic function on $\Omega$. Recall that for the Green function $G_{\Omega}(\cdot,z_0)$, there exists a $\chi_{z_0}$ and a multiplicative function $f_{z_0}\in\mathcal{O}^{\chi_{z_0}}(\Omega)$ such that $|f_{z_0}(z)|=P^*(e^{G_{\Omega}(z,z_0)})$ (see \cite{suita72}).

	 Let $\tilde Z_0=\{z_j:j\in \mathbb{Z}_{\ge 0}\, \&\, 1\leq j\leq m\}$ be a finite subset of the open Riemann surface $\Omega$.
	Let $w_j$ be a local coordinate on a neighborhood $V_{z_j}\subset\subset\Omega$ of $z_j$ satisfying $w_j(z_j)=0$ for $z_j\in \tilde{Z}_0$, where $V_{z_j}\cap V_{z_k}=\emptyset$ for any  $j\neq k$. Denote that $V_0:=\bigcup_{1\leq j\leq m}V_{z_j}$. Without loss of generality, assume that $w_j(V_{z_j})=\{w\in\mathbb{C}:|w|<r_j\}$, where $r_j>0$. Let $f$ be a holomorphic $(n,0)$ form on $V_0\times Y$.
Theorem \ref{thm:concavity} implies that $G(h^{-1}(r))$ is concave with respect to $r$, where $h(t)=\int_t^{+\infty}c(s)e^{-s}ds$. 	We recall a characterization of the concavity of $G(h^{-1}(r))$ degenerating to linearity for the fibers over sets of finite points as follows.

	\begin{Theorem}[\cite{BGY-concavity5}]\label{finite-p}
		Assume that $e^{-\varphi}c(-\psi)$ has a positive lower bound on any compact subset of $M\setminus Z_0$, $G(0)\in (0,+\infty)$ and $(\psi_1-2q_{z_j}G_{\Omega}(\cdot,z_j))(z_j)>-\infty$, where $q_{z_j}=\frac{1}{2}v(dd^c(\psi_1),z_j)>0$ for any $j\in\{1,2,\ldots,m\}$. Then $G(h^{-1}(r))$ is linear with respect to $r\in (0,\int_0^{+\infty}c(t)e^{-t}dt]$ if and only if the following statements hold:
		
		$(1)$ $\psi_1=2\sum\limits_{j=1}^mq_{z_j} G_{\Omega}(\cdot,z_j)$;
		
		$(2)$ for any $j\in\{1,2,\ldots,m\}$, $f=\pi_1^*(a_jw_j^{k_j}dw_j)\wedge \pi_2^*(f_{Y})+f_j$ on $V_{z_j}\times Y$, where $a_j\in\mathbb{C}\setminus \{0\}$ is a constant, $k_j$ is a nonnegative integer, $f_{Y}$ is a holomorphic $(n-1,0)$ form on $Y$ such that $\int_{Y}|f_{Y}|^2e^{-\varphi_{Y}}\in (0,+\infty)$, and $(f_j,(z_j,y))\in \mathcal{O}(K_M)_{(z_j,y)}\otimes\mathcal{I}(\varphi+\psi)_{(z_j,y)}$ for any $j\in\{1,2,\ldots,m\}$ and $y\in Y$;
		
		$(3)$ $\varphi_{\Omega}+\psi_1=2\log |g|+2\sum\limits_{j=1}^mG_{\Omega}(\cdot,z_j)+2u$, where $g$ is a holomorphic function on $\Omega$ such that $ord_{z_j}(g)=k_j$ and $u$ is a harmonic function on $\Omega$;
		
		$(4)$ $\prod\limits_{j=1}^m\chi_{z_j}=\chi_{-u}$, where $\chi_{-u}$ and $\chi_{z_j}$ are the characters associated to the functions $-u$ and $G_{\Omega}(\cdot,z_j)$ respectively;
		
		$(5)$ for any $j\in\{1,2,\ldots,m\}$,
		\begin{equation*}
			\lim_{z\rightarrow z_j}\frac{a_jw_j^{k_j}dw_j}{gP_*\left(f_u\left(\prod\limits_{l=1}^mf_{z_l}\right)\left(\sum\limits_{l=1}^mp_l\dfrac{d{f_{z_{l}}}}{f_{z_{l}}}\right)\right)}=c_0,
		\end{equation*}
		where $c_0\in\mathbb{C}\setminus\{0\}$ is a constant independent of $j$.
	\end{Theorem}

	\begin{Remark}[see \cite{GMY-boundary3}]\label{r:equivalent}
	The statements $(3)$, $(4)$ and $(5)$ hold if and only if the following statements hold:
	
		$(1)'$ $\varphi_{\Omega}+\psi_1=2\log|g_1|$, where $g_1$ is a holomorphic function on $\Omega$ such that $ord_{z_j}(g_1)=k_j+1$ for any $j\in\{1,2,\ldots,m\}$;

	$(2)'$  $\frac{ord_{z_j}(g_1)}{q_{z_j}}\lim_{z\rightarrow z_j}\frac{a_jw_j^{k_j}dw_j}{dg_1}=c_0$ for any $j\in\{1,2,\ldots,m\}$, where $c_0\in\mathbb{C}\backslash\{0\}$ is a constant independent of $j$;
\end{Remark}
		
		We give a generalization of Theorem \ref{finite-p}, which will be used in the proofs of Proposition \ref{p:n-linearity1} and Theorem \ref{thm:fibra-finite}.
		\begin{Proposition}
			\label{p-finite-2}Let $G(0)\in (0,+\infty)$ and $q_{z_j}=\frac{1}{2}v(dd^c(\psi_1),z_j)>0$ for any $j\in\{1,2,\ldots,m\}$. For any $j\in\{1,2,\ldots,m\}$, assume that one of the following conditions holds:
			
			$(A)$ $\varphi_{\Omega}+a\psi_1$ is subharmonic near $z_j$ for some $a\in[0,1)$;
			
			$(B)$ $(\psi_1-2q_{z_j}G_{\Omega}(\cdot,z_j))(z_j)>-\infty.$
			
		 Then $G(h^{-1}(r))$ is linear with respect to $r\in (0,\int_0^{+\infty}c(t)e^{-t}dt]$ if and only if the following statements hold:
		
		$(1)$ $\psi_1=2\sum\limits_{j=1}^mq_{z_j} G_{\Omega}(\cdot,z_j)$;
		
		$(2)$ for any $j\in\{1,2,\ldots,m\}$, $f=\pi_1^*(a_jw_j^{k_j}dw_j)\wedge \pi_2^*(f_{Y})+f_j$ on $V_{z_j}\times Y$, where $a_j\in\mathbb{C}\setminus \{0\}$ is a constant, $k_j$ is a nonnegative integer, $f_{Y}$ is a holomorphic $(n-1,0)$ form on $Y$ such that $\int_{Y}|f_{Y}|^2e^{-\varphi_{Y}}\in (0,+\infty)$, and $(f_j,(z_j,y))\in \mathcal{O}(K_M)_{(z_j,y)}\otimes\mathcal{I}(\varphi+\psi)_{(z_j,y)}$ for any $j\in\{1,2,\ldots,m\}$ and $y\in Y$;
		
		$(3)$ $\varphi_{\Omega}+\psi_1=2\log|g_1|$, where $g_1$ is a holomorphic function on $\Omega$ such that $ord_{z_j}(g_1)=k_j+1$ for any $j\in\{1,2,\ldots,m\}$;
		
		$(4)$ $\frac{ord_{z_j}(g_1)}{q_{z_j}}\lim_{z\rightarrow z_j}\frac{a_jw_j^{k_j}dw_j}{dg_1}=c_0$ for any $j\in\{1,2,\ldots,m\}$, where $c_0\in\mathbb{C}\backslash\{0\}$ is a constant independent of $j$.
		\end{Proposition}
		\begin{proof}
			The sufficiency follows from Theorem \ref{finite-p} and Remark \ref{r:equivalent}. Thus, we just need to prove the necessity. It follows form Theorem \ref{finite-p} and Remark \ref{r:equivalent} that it suffices to prove $\psi_1=2\sum\limits_{j=1}^mq_{z_j} G_{\Omega}(\cdot,z_j)$.
			
			It follows from Corollary  \ref{c:necessary condition for linear of G} that there exists a holomorphic $(n,0)$ form $F$ on $M$ satisfying $(F-f,p)\in(\mathcal{O}(K_{M})\otimes\mathcal{I}(\varphi+\psi))_p$ for any $p\in Z_0$,
					\begin{equation}
				\label{eq:0407c}G(t)=\int_{\{\psi<-t\}}|F|^2e^{-\varphi}c(-\psi)
			\end{equation}
			 for any $t\ge0$ and
			$$\int_{M}|F|^2e^{-\varphi}a(-\psi)=\frac{G(0)}{\int_0^{+\infty}c(s)e^{-s}ds}\int_0^{+\infty}a(s)e^{-s}ds$$
			for any nonnegative measurable function $a(t)$ on $(0,+\infty)$.		It follows from Lemma \ref{decomp} that $F=\sum_{l\ge\tilde k_j}\pi_1^*(w^l_jdw_j)\wedge\pi_2^*(F_{j,l})$ on $V_{z_j}\times Y$ for any $j\in\{1,2,\ldots,m\}$, where $\tilde k_j$ is a nonnegative integer and $F_{j,l}$ is a holomorphic $(n-1,0)$ form on $Y$ satisfying that $\tilde F_j:=F_{j,\tilde k_j}\not\equiv0$. Denote that
		\begin{equation*}
			I_F:=\{j:1\leq j\le m \, \& \, \tilde{k}_j+1-k_j\leq 0\}.
		\end{equation*}
		Using the Weierstrass Theorem on open Riemann surfaces (see \cite{OF81}) and
	 Siu's Decomposition Theorem, we have
	\[\varphi_{\Omega}+\psi_1=2\log |g_0|+2u_0,\]
	where $g_0$ is a holomorphic function on $\Omega$ and $u_0$ is a subharmonic function on $\Omega$ such that $v(dd^cu_0,z)\in [0,1)$ for any $z\in\Omega$. There exists a nonnegative integer $k_j$ such that $d_j:=\lim_{z\rightarrow z_j}\frac{g_0(z)}{w_j^{k_j}(z)}\in\mathbb{C}\backslash\{0\}$.
		 It follows from Lemma \ref{l:inte} that  $\tilde{k}_j+1-k_j=0$ for any $j\in I_F$, $\int_Y| F_{j,l}|^2e^{-\varphi_Y}<+\infty$ for any $j,l$, and
		\begin{equation}\label{eq:0407b}
			\sum_{j\in I_F}\frac{2\pi e^{-2u_0(z_j)}}{q_{z_j}|d_j|^2}\int_Y|\tilde{F}_j|^2e^{-\varphi_Y}\le \frac{G(0)}{\int_0^{+\infty}c(s)e^{-s}ds}.
		\end{equation}
Note that $|w_j^{l}|^2e^{-\varphi_{\Omega}-\psi_1}$ is integrable near $z_j$ for any $l\ge k_j$ and any $j\in\{1,2,\ldots,m\}$, hence it follows from Lemma \ref{closedness} that $(\sum_{l\ge k_j}\pi_1^*(w^l_jdw_j)\wedge\pi_2^*(F_{j,l}),p)\in(\mathcal{O}(K_M)\otimes\mathcal{I}(\varphi+\psi))_p$ for any $p\in Z_0$.
Denote that $\tilde \psi_1:=2\sum\limits_{j=1}^mq_{z_j} G_{\Omega}(\cdot,z_j)$ and $\tilde\varphi_{\Omega}:=\varphi_{\Omega}+\psi_1-\tilde\psi_1$. Note that $\tilde\psi_1\ge \psi_1$, $\tilde\varphi_{\Omega}+\tilde\psi_1=\varphi_{\Omega}+\psi_1$ and $c(t)e^{-t}$ is decreasing on $(0,+\infty)$.
 Using Theorem \ref{L2ext-finite-f}, there exists a holomorphic $(n,0)$ form $\tilde F$ on $M$, such that $(\tilde F-F,p)\in(\mathcal{O}(K_M)\otimes\mathcal{I}(\pi_1^*(\tilde\varphi_{\Omega}+\tilde\psi_1)+\pi_2^*(\varphi_Y)))_p=(\mathcal{O}(K_M)\otimes\mathcal{I}(\varphi+\psi))_p$ for any $p\in Z_0$ and
\begin{equation}
	\label{eq:0407a}
	\begin{split}
	G(0)&\le\int_{M}|\tilde F|^2e^{-\varphi}c(-\psi)\\
	&\le\int_{M}|\tilde F|^2e^{-\pi_1^*(\tilde\varphi_{\Omega})-\pi_2^*(\varphi_Y)}c(-\pi_1^*(\tilde\psi_1)) \\
	&\le\left(\int_0^{+\infty}c(s)e^{-s}ds\right)\sum_{j\in I_F}\frac{2\pi e^{-2u_0(z_j)}}{q_{z_j}|d_j|^2}\int_Y|\tilde{F}_j|^2e^{-\varphi_Y}.\end{split}
\end{equation}
Combining equality \eqref{eq:0407c}, inequality \eqref{eq:0407b} and inequality \eqref{eq:0407a}, we obtain that
\begin{equation*}
	\int_{M}|\tilde F|^2e^{-\varphi}c(-\psi)=\int_{M}|\tilde F|^2e^{-\pi_1^*(\tilde\varphi_{\Omega})-\pi_2^*(\varphi_Y)}c(-\pi_1^*(\tilde\psi_1)).
\end{equation*}
As $\tilde F\not\equiv0$, it follows from Lemma \ref{l:psi=G} that $\psi_1=\tilde\psi_1=2\sum\limits_{j=1}^mq_{z_j} G_{\Omega}(\cdot,z_j).$

Thus, Proposition \ref{p-finite-2} holds.
		\end{proof}

				Let $\tilde Z_0=\{z_j:j\in \mathbb{Z}_{\ge1}\}$ be an infinite discrete subset of the open Riemann surface $\Omega$. Let $w_j$ be a local coordinate on a neighborhood $V_{z_j}\subset\subset\Omega$ of $z_j$ satisfying $w_j(z_j)=0$ for $z_j\in \tilde Z_0$, where $V_{z_j}\cap V_{z_k}=\emptyset$ for any  $j\neq k$. Without loss of generality, assume that $w_j(V_{z_j})=\{w\in\mathbb{C}:|w|<r_j\}$, where $r_j>0$. Denote that $V_0:=\bigcup_{j=1}^{\infty}V_{z_j}$. Let $f$ be a holomorphic $(n,0)$ form on $V_0\times Y$.
We recall a necessary condition such that $G(h^{-1}(r))$ is linear for the fibers over infinite analytic subsets as follows.
	
	\begin{Proposition}[\cite{BGY-concavity5}]\label{infinite-p}
		Assume that $e^{-\varphi}c(-\psi)$ has a positive lower bound on any compact subset of $M\setminus Z_0$, $G(0)\in (0,+\infty)$ and $(\psi_1-2q_{z_j}G_{\Omega}(\cdot,z_j))(z_j)>-\infty$, where $q_{z_j}=\frac{1}{2}v(dd^c(\psi_1),z_j)>0$ for any $j\in\mathbb{Z}_{\ge1}$. Assume that $G(h^{-1}(r))$ is linear with respect to $r\in (0,\int_0^{+\infty}c(t)e^{-t}dt]$, then the following statements hold:
		
		$(1)$ $\psi_1=2\sum\limits_{j=1}^{\infty}q_{z_j} G_{\Omega}(\cdot,z_j)$;
		
		$(2)$ for any $j\in\mathbb{Z}_{\ge1}$, $f=\pi_1^*(a_jw_j^{k_j}dw_j)\wedge \pi_2^*(f_{Y})+f_j$ on $V_{z_j}\times Y$, where $a_j\in\mathbb{C}\setminus \{0\}$ is a constant, $k_j$ is a nonnegative integer, $f_{Y}$ is a holomorphic $(n-1,0)$ form on $Y$ such that $\int_{Y}|f_{Y}|^2e^{-\varphi_{Y}}\in (0,+\infty)$, and $(f_j,(z_j,y))\in \mathcal{O}(K_M)_{(z_j,y)}\otimes\mathcal{I}(\varphi+\psi)_{(z_j,y)}$ for any $j\in\mathbb{Z}_{\ge1}$ and $y\in Y$;
		
		$(3)$ $\varphi_{\Omega}+\psi_1=2\log |g|$, where $g$ is a holomorphic function on $\Omega$ such that $ord_{z_j}(g)=k_j+1$ for any $j\in\mathbb{Z}_{\ge1}$;
		
		$(4)$ for any $j\in\mathbb{Z}_{\ge1}$,
		\begin{equation*}
			\frac{q_{z_j}}{ord_{z_j}(g)}\lim_{z\rightarrow z_j}\frac{dg}{a_jw_j^{k_j}dw_j}=c_0,
		\end{equation*}
		where $c_0\in\mathbb{C}\setminus\{0\}$ is a constant independent of $j$;
		
		$(5)$ $\sum\limits_{j\in\mathbb{Z}_{\ge1}}q_{z_j}<+\infty$.
	\end{Proposition}

	We give a generalization of Proposition \ref{infinite-p}, which will be used in the proof of Proposition \ref{p:n-linearity1}.
		\begin{Proposition}
			\label{p-infinite-2}Let $G(0)\in (0,+\infty)$ and $q_{z_j}=\frac{1}{2}v(dd^c(\psi_1),z_j)>0$ for any $j\in\mathbb{Z}_{\ge1}$. For any $j\in\mathbb{Z}_{\ge1}$, assume that one of the following conditions holds:
			
			$(A)$ $\varphi_{\Omega}+a\psi_1$ is subharmonic near $z_j$ for some $a\in[0,1)$;
			
			$(B)$ $(\psi_1-2q_{z_j}G_{\Omega}(\cdot,z_j))(z_j)>-\infty.$
			
		 If $G(h^{-1}(r))$ is linear with respect to $r\in (0,\int_0^{+\infty}c(t)e^{-t}dt]$,  then the following statements hold:
		
		$(1)$ $\psi_1=2\sum\limits_{j=1}^{\infty}q_{z_j} G_{\Omega}(\cdot,z_j)$;
		
		$(2)$ for any $j\in\mathbb{Z}_{\ge1}$, $f=\pi_1^*(a_jw_j^{k_j}dw_j)\wedge \pi_2^*(f_{Y})+f_j$ on $V_{z_j}\times Y$, where $a_j\in\mathbb{C}\setminus \{0\}$ is a constant, $k_j$ is a nonnegative integer, $f_{Y}$ is a holomorphic $(n-1,0)$ form on $Y$ such that $\int_{Y}|f_{Y}|^2e^{-\varphi_{Y}}\in (0,+\infty)$, and $(f_j,(z_j,y))\in \mathcal{O}(K_M)_{(z_j,y)}\otimes\mathcal{I}(\varphi+\psi)_{(z_j,y)}$ for any $j\in\mathbb{Z}_{\ge1}$ and $y\in Y$;
		
		$(3)$ $\varphi_{\Omega}+\psi_1=2\log |g|$, where $g$ is a holomorphic function on $\Omega$ such that $ord_{z_j}(g)=k_j+1$ for any $j\in\mathbb{Z}_{\ge1}$;
		
		$(4)$ for any $j\in\mathbb{Z}_{\ge1}$,
		\begin{equation*}
			\frac{q_{z_j}}{ord_{z_j}(g)}\lim_{z\rightarrow z_j}\frac{dg}{a_jw_j^{k_j}dw_j}=c_0,
		\end{equation*}
		where $c_0\in\mathbb{C}\setminus\{0\}$ is a constant independent of $j$;
		
		$(5)$ $\sum\limits_{j\in\mathbb{Z}_{\ge1}}q_{z_j}<+\infty$.		\end{Proposition}
		\begin{proof}
		The proof of Proposition \ref{p-infinite-2} is similar to Proposition \ref{p-finite-2}.
		
			It follows form Proposition \ref{infinite-p} that it suffices to prove $\psi_1=2\sum\limits_{j=1}^{+\infty}q_{z_j} G_{\Omega}(\cdot,z_j)$.
			
			It follows from Corollary  \ref{c:necessary condition for linear of G} that there exists a holomorphic $(n,0)$ form $F$ on $M$ satisfying $(F-f,p)\in(\mathcal{O}(K_{M})\otimes\mathcal{I}(\varphi+\psi))_p$ for any $p\in Z_0$,
					\begin{equation}
				\label{eq:0407d}G(t)=\int_{\{\psi<-t\}}|F|^2e^{-\varphi}c(-\psi)
			\end{equation}
			 for any $t\ge0$ and
			$$\int_{M}|F|^2e^{-\varphi}a(-\psi)=\frac{G(0)}{\int_0^{+\infty}c(s)e^{-s}ds}\int_0^{+\infty}a(s)e^{-s}ds$$
			for any nonnegative measurable function $a(t)$ on $(0,+\infty)$.		It follows from Lemma \ref{decomp} that $F=\sum_{l\ge\tilde k_j}\pi_1^*(w^l_jdw_j)\wedge\pi_2^*(F_{j,l})$ on $V_{z_j}\times Y$ for any $j\in\mathbb{Z}_{\ge1}$, where $\tilde k_j$ is a nonnegative integer and $F_{j,l}$ is a holomorphic $(n-1,0)$ form on $Y$ satisfying that $\tilde F_j:=F_{j,\tilde k_j}\not\equiv0$. Denote that
		\begin{equation*}
			I_F:=\{j:j\in\mathbb{Z}_{\ge1} \, \& \, \tilde{k}_j+1-k_j\leq 0\}.
		\end{equation*}
		Using the Weierstrass Theorem on open Riemann surfaces (see \cite{OF81}) and
 Siu's Decomposition Theorem, we have
	\[\varphi_{\Omega}+\psi_1=2\log |g_0|+2u_0,\]
	where $g_0$ is a holomorphic function on $\Omega$ and $u_0$ is a subharmonic function on $\Omega$ such that $v(dd^cu_0,z)\in [0,1)$ for any $z\in\Omega$. There exists a nonnegative integer $k_j$ such that $d_j:=\lim_{z\rightarrow z_j}\frac{g_0(z)}{w_j^{k_j}(z)}\in\mathbb{C}\backslash\{0\}$.
		 It follows from Lemma \ref{l:inte} that  $\tilde{k}_j+1-k_j=0$ for any $j\in I_F$, $\int_Y| F_{j,l}|^2e^{-\varphi_Y}<+\infty$ for any $j,l$, and
		\begin{equation}\label{eq:0407e}
			\sum_{j\in I_F}\frac{2\pi e^{-2u_0(z_j)}}{q_{z_j}|d_j|^2}\int_Y|\tilde{F}_j|^2e^{-\varphi_Y}\le \frac{G(0)}{\int_0^{+\infty}c(s)e^{-s}ds}.
		\end{equation}
Note that $|w_j^{l}|^2e^{-\varphi_{\Omega}-\psi_1}$ is integrable near $z_j$ for any $l\ge k_j$ and any $j\in\mathbb{Z}_{\ge1}$, hence it follows from Lemma \ref{closedness} that $(\sum_{l\ge k_j}\pi_1^*(w^l_jdw_j)\wedge\pi_2^*(F_{j,l}),p)\in(\mathcal{O}(K_M)\otimes\mathcal{I}(\varphi+\psi))_p$ for any $p\in Z_0$.
Denote that $\tilde \psi_1:=2\sum\limits_{j=1}^{+\infty}q_{z_j} G_{\Omega}(\cdot,z_j)$ and $\tilde\varphi_{\Omega}:=\varphi_{\Omega}+\psi_1-\tilde\psi_1$. Note that $\tilde\psi_1\ge \psi_1$, $\tilde\varphi_{\Omega}+\tilde\psi_1=\varphi_{\Omega}+\psi_1$ and $c(t)e^{-t}$ is decreasing on $(0,+\infty)$.
 Using Theorem \ref{L2ext-finite-f}, there exists a holomorphic $(n,0)$ form $\tilde F$ on $M$, such that $(\tilde F-F,p)\in(\mathcal{O}(K_M)\otimes\mathcal{I}(\pi_1^*(\tilde\varphi_{\Omega}+\tilde\psi_1)+\pi_2^*(\varphi_Y)))_p=(\mathcal{O}(K_M)\otimes\mathcal{I}(\varphi+\psi))_p$ for any $p\in Z_0$ and
\begin{equation}
	\label{eq:0407f}
	\begin{split}
	G(0)&\le\int_{M}|\tilde F|^2e^{-\varphi}c(-\psi)\\
	&\le\int_{M}|\tilde F|^2e^{-\pi_1^*(\tilde\varphi_{\Omega})-\pi_2^*(\varphi_Y)}c(-\pi_1^*(\tilde\psi_1)) \\
	&\le\left(\int_0^{+\infty}c(s)e^{-s}ds\right)\sum_{j\in I_F}\frac{2\pi e^{-2u_0(z_j)}}{q_{z_j}|d_j|^2}\int_Y|\tilde{F}_j|^2e^{-\varphi_Y}.\end{split}
\end{equation}
Combining equality \eqref{eq:0407d}, inequality \eqref{eq:0407e} and inequality \eqref{eq:0407f}, we obtain that
\begin{equation*}
	\int_{M}|\tilde F|^2e^{-\varphi}c(-\psi)=\int_{M}|\tilde F|^2e^{-\pi_1^*(\tilde\varphi_{\Omega})-\pi_2^*(\varphi_Y)}c(-\pi_1^*(\tilde\psi_1)).
\end{equation*}
As $\tilde F\not\equiv0$, it follows from Lemma \ref{l:psi=G} that $\psi_1=\tilde\psi_1=2\sum\limits_{j=1}^{+\infty}q_{z_j} G_{\Omega}(\cdot,z_j).$

Thus, Proposition \ref{p-infinite-2} holds.
		\end{proof}

\section{A necessary condition for $G(h^{-1}(r))$ is linearity}\label{sec:n}

Let $\Omega$  be an open Riemann surface, which admits a nontrivial Green function $G_{\Omega}$. Let $Y$ be an $n-1$ dimensional weakly pseudoconvex K\"ahler manifold, and let $K_Y$ be the canonical  line bundle on $Y$. Let $M=\Omega \times Y$ be an $n-$dimensional complex manifold. Let $\pi_{1}$ and $\pi_2$ be the natural projections from $M$ to $\Omega$ and $Y$ respectively. Let $K_M$ be the canonical line bundle on $M$.

  Let $\psi$ be a subharmonic function on $\Omega$. Let $\varphi_\Omega$ be a Lebesgue measurable function on $\Omega$ such that $\varphi_\Omega+\psi$ is subharmonic function on $\Omega$. Let $F$ be a holomorphic function on $\Omega$. Let $T\in [-\infty,+\infty)$.
Denote that
$$\tilde{\Psi}:=\min\{\psi-2\log|F|,-T\}.$$
For any $z \in \Omega$ satisfying $F(z)=0$,
we set $\tilde{\Psi}(z)=-T$. Denote that $\Psi:=\pi_1^*(\tilde\Psi)$ on $M$. Let $\varphi_Y$ be a plurisubharmonic function on $Y$. Denote that $\varphi:=\pi_1^*(\varphi_{\Omega})+\pi_2^*(\varphi_Y)$.

 Let ${Z}_0\subset M$ be a subset of $\cap_{t>T} \overline{\{\Psi<-t\}}$ such that there exists a subset $\tilde Z_0$ of $\Omega$ such that  $Z_0=\tilde{Z}_0\times Y$. Denote that $\tilde{Z}_1:=\{z\in \tilde{Z}_0:v(dd^c(\psi),z)\ge2ord_{z}(F)\}$ and $\tilde{Z}_2:=\{z\in \tilde{Z}_0:v(dd^c(\psi),z)<2ord_{z}(F)\},$ where $d^c=\frac{\partial-\bar\partial}{2\pi\sqrt{-1}}$ and $v(dd^c(\psi),z)$ is the Lelong number of $dd^c(\psi)$ at $z$ (see \cite{Demaillybook}).
  Denote that  $\tilde{Z}_3:=\{z\in \tilde{Z}_0:v(dd^c(\psi),z)>2ord_{z}(F)\}$. Note that $\{\tilde{\Psi}<-t\}\cup \tilde{Z}_3$ is an open Riemann surface for any $t\ge T$. Denote $Z_1:=\tilde{Z}_1\times Y$, $Z_2:=\tilde{Z}_2\times Y$ and $Z_3:=\tilde{Z}_3\times Y$ respectively.

Let $c(t)$ be a positive function on $(0,+\infty)$ such that $c(t)e^{-t}$ is decreasing on $(T,+\infty)$, $c(t)e^{-t}$ is integrable near $+\infty$, and $c(-\Psi)e^{-\varphi}$ has a positive lower bound on $K\cap\{\Psi<-T\}$ for any compact subset $K\subset M\backslash \pi_1^{-1}(E)$, where $E$  is an analytic subset of $\Omega$ such that $E\subset\{\tilde\Psi=-\infty\}$.

Let $f$ be a holomorphic $(n,0)$ form on $\{\Psi<-t_0\}\cap V$, where $V\supset Z_0$ is an open subset of $M$ and $t_0>T$
is a real number.
Denote
\begin{equation}
\label{def of g(t) for boundary pt}
\begin{split}
\inf\Bigg\{ \int_{ \{ \Psi<-t\}}|\tilde{f}|^2&e^{-\varphi}c(-\Psi): \tilde{f}\in
H^0(\{\Psi<-t\},\mathcal{O} (K_{M})  ) \\
&\&\, (\tilde{f}-f)_{p}\in
\mathcal{O} (K_{M})_{p} \otimes I(\varphi+\Psi)_{p}\text{ for any }  p\in Z_0 \Bigg\}
\end{split}
\end{equation}
by $G(t;c,\Psi,\varphi,I(\varphi+\Psi),f)$, where $t\in[T,+\infty)$.
Without misunderstanding, we denote $G(t;c,\Psi,\varphi,I(\varphi+\Psi),f)$ by $G(t)$ for simplicity.

Note that there exists a subharmonic function $\psi_1$ on $\Omega_1:= \{\tilde\Psi<-T\}\cup \tilde Z_3$ such that $\psi_1+2\log|F|=\psi$. Denote that $M':=\Omega_1\times Y\subset M$ and $\tilde\psi:=\pi_1^*(\psi_1)$.

For any $z\in \Omega_1$, if $F(z)=0$, for any $y\in Y$, we have $\Psi((z,y))\not=-\infty$, then we know that $e^{-\varphi}c(-\Psi)=e^{-\varphi}c(-\tilde\psi)$ has a positive lower bound on $K\backslash (\{z\}\times Y)\subset K\cap\{\Psi<-T\}$, where $K\Subset M$ is a neighborhood of $(z,y)$.
Combining $e^{-\varphi}c(-\tilde\psi)\le Ce^{-\varphi-\tilde\psi}=Ce^{-\pi_1^*(\varphi_{\Omega}+\psi-2\log|F|)-\pi_2^*(\varphi_Y)}$ on $K$, we have $v(dd^c(\varphi_{\Omega}+\psi),z)\ge 2ord_{z}(F)$. Hence we have $\varphi_{\Omega}+\psi_1$ is a subharmonic function on $\Omega_1$.

For any $z\in \tilde Z_3$, if $F(z)\not=0$, we know that  $I(\varphi+\Psi)_{(z,y)}=\mathcal{I}(\varphi+\tilde \psi)_{(z,y)}$ is an ideal of $\mathcal{O}_{M,(z,y)}$ for any $y\in Y$. For any $z\in\tilde Z_3$, if $F(z)=0$, we know that  $h_{(z,y)}\in I(\varphi+\Psi)_{(z,y)}$ if and only if there exists a holomorphic extension $\tilde h$ of $h$ near $(z,y)$ such that $(\tilde h,(z,y))\in \mathcal{I}(\varphi+\tilde\psi)_{(z,y)}$, where   $y\in Y$ and  $h_{(z,y)}\in I(\varphi+\Psi)_{(z,y)}$.

Let $f_1$ be a holomorphic $(n,0)$ form on $V_1$, where $V_1\supset Z_3$ is an open subset of $M'$. Denote
\begin{equation*}
\begin{split}
\inf\Bigg\{ \int_{ \{ \tilde\psi<-t\}}|\tilde{f}|^2&e^{-\varphi}c(-\tilde\psi): \tilde{f}\in
H^0(\{\tilde\psi<-t\},\mathcal{O} (K_{M})  ) \\
&\&\, (\tilde{f}-f_1,p)\in
\mathcal{O} (K_{M})_{p} \otimes \mathcal{I}(\varphi+\tilde\psi)_p,\text{ for any }  p\in Z_3 \Bigg\}
\end{split}
\end{equation*}
by $G_{f_1}(t)$, where $t\in[T,+\infty)$.

Note that $e^{-\varphi}c(-\tilde\psi)$ has a positive lower bound on any compact subset of $M'\backslash\big(\pi_1^{-1}(E\cup \tilde Z_3)\big)$ and $E\cup \tilde Z_3\subset\{\psi_1=-\infty\}$. Let $f_2$ be a holomorphic $(n,0)$ form on $V\cap(\{\Psi<-t_0\}\cup Z_3)=V\cap\{\tilde\psi<-t_0\}$ such that $(f_2)_p\in\mathcal{O}(K_{M})_p\otimes H_p$ for any $p\in Z_0$, where $H_p=\{h_p\in J(\Psi)_p:\int_{\{\Psi<-t\}\cap U}|h|^2e^{-\varphi}c(-\Psi)dV_M<+\infty$ for some $t>T$ and some neighborhood $U$ of $p\}$, $dV_M$ is a continuous volume form on $M$, $V\supset Z_0$ is an open subset of $M$ and $t_0>T$
is a real number. Theorem \ref{thm:concavity} shows that
$G(h^{-1}(r);c,\Psi,\varphi,I(\varphi+\Psi),f_2)$ and $G_{f_2}(h^{-1}(r))$ are concave with respect to $r$, where  $h(t)=\int_{t}^{+\infty}c(s)e^{-s}ds$. We give a relationship between $G(h^{-1}(r);c,\Psi,\varphi,I(\varphi+\Psi),f_2)$ and $G_{f_2}(h^{-1}(r))$, which will be used in the proof of Proposition \ref{p:n-linearity1}.

\begin{Lemma}\label{l:inner}
	If $H_p=I(\varphi+\Psi)_p$ for any $p\in Z_0\backslash Z_3$, then $G(t;c,\Psi,\varphi,I(\varphi+\Psi),f_2)=\tilde G_{f_2}(t)$ for any $t\ge T$.
\end{Lemma}

\begin{proof}
	 For any $t\ge T$ and holomorphic $(n,0)$ form $\tilde f$ on $\{\tilde\psi<-t\}$ satisfying $(\tilde f-f_2,p)\in
\mathcal{O} (K_{M})_{p} \otimes \mathcal{I}(\varphi+\tilde\psi)_p$ for any   $p\in Z_3 $ and $\int_{\{\tilde\psi<-t\}}|\tilde f|^2e^{-\varphi}c(-\tilde\psi)<+\infty$, it follows from $(f_2)_p\in
\mathcal{O} (K_{M})_{p} \otimes H_{p}$ for any $p\in Z_0$ and $H_p=I(\varphi+\Psi)_p$ for any $p\in Z_0\backslash Z_3$ that $(\tilde f-f_2)_p\in
\mathcal{O} (K_{M})_{p} \otimes I(\varphi+\Psi)_p$ for any $z\in Z_0$.
As $\mu(Z_3)=0$, where $\mu$ is the Lebesgue measure on $M$,   the definitions of $G(t;c,\Psi,\varphi,I(\varphi+\Psi),f_2)$ and $G_{f_2}(t)$ show that $G(t;c,\Psi,\varphi,I(\varphi+\Psi),f_2)\le\tilde G_{f_2}(t)$ for any $t\ge T$.
	
	For any $t\ge T$, let  $\tilde f$ be a holomorphic $(n,0)$ form on $\{\Psi<-t\}$ satisfying $(\tilde f-f_2)_{p}\in
\mathcal{O} (K_{M})_{p} \otimes I(\varphi+\Psi)_p$ for any $z\in Z_0$ and $\int_{\{\Psi<-t\}}|\tilde f|^2e^{-\varphi}c(-\Psi)<+\infty$.
For any $p=(z,y)\in  Z_3\backslash\{\Psi<-t\}$, we have $F(z)=0$. Note that $v(dd^c(\psi),z)>2ord_{z}(F)$, then $e^{-\varphi}c(-\Psi)$ has a positive lower bound on $K\backslash(\{z\}\times Y) \subset\{\Psi<-t\}$, where $K\Subset M$ is a neighborhood  of $(z,y)$. Following from $\int_{\{\Psi<-t\}}|\tilde f|^2e^{-\varphi}c(-\Psi)<+\infty$, we get that there exists a holomorphic $(n,0)$ form $\tilde f_1$ on $\{\tilde\psi<-t\}=\{\Psi<-t\}\cup Z_3$ such that $\tilde f_1=\tilde f$ on $\{\Psi<-t\}$, which implies that $(\tilde f_1-f_2,p)\in\mathcal{O}(K_{M})_p\otimes \mathcal{I}(\varphi+\tilde\psi)_p$ for any $p\in Z_3$ and $\int_{\{\tilde\psi<-t\}}|\tilde f_1|^2e^{-\varphi}c(-\tilde\psi)=\int_{\{\Psi<-t\}}|\tilde f|^2e^{-\varphi}c(-\Psi)$. By the definitions of $G(t;c,\Psi,\varphi,I(\varphi+\Psi),f_2)$ and $G_{f_2}(t)$, we have  $G(t;c,\Psi,\varphi,I(\varphi+\Psi),f_2)\ge\tilde G_{f_2}(t)$ for any $t\ge T$.
	
	Thus, Lemma \ref{l:inner} holds.
\end{proof}

Note that $\tilde Z_3$ is a discrete subset of $\Omega$.
Let $\tilde Z_3=\{z_j:1\le j<\gamma\}$, where $\gamma\in\mathbb{Z}_{\ge2}\cup\{+\infty\}$. Let $w_j$ be a local coordinate on a neighborhood $V_{z_j}\Subset\Omega$ of $z_j$ satisfying that $w_j(z_j)=0$ for any $z_j\in \tilde Z_3$, where $V_{z_j}\cap V_{z_k}=\emptyset$ for any $j\not=k$.
We give a necessary condition for the concavity of $G(h^{-1}(r))$ degenerating to linearity.
\begin{Proposition}
	\label{p:n-linearity1}
For any $z\in \tilde Z_1$, assume that one of the following conditions holds:
	
	$(A)$ $\varphi_{\Omega}+a\psi$ is subharmonic near $z$ for some $a\in[0,1)$;
	
	$(B)$ $(\psi-2q_z\log|w|)(z)>-\infty$, where $q_z=\frac{1}{2}v(dd^c(\psi),z)$ and $w$ is a local coordinate on a neighborhood of $z$ satisfying that $w(z)=0$.
	
	If there exists $t_1\ge T$ such that $G(t_1)\in(0,+\infty)$ and $G(h^{-1}(r))$ is linear with respect to $r\in(0,\int_T^{+\infty}c(s)e^{-s}ds)$, then the following statements hold:
	
	$(1)$ $f=\pi_1^*(a_jw_j^{k_j}dw_j)\wedge\pi_2^*(f_Y)+f_j$ on $(V_{z_j}\times Y)\cap\{\Psi<-t_0\}\cap V$, where $a_j\in\mathbb{C}\backslash\{0\}$, $k_j$ is a nonnegative integer, $f_Y$ is a holomorphic $(n-1,0)$ form on $Y$ satisfying $\int_Y|f_Y|^2e^{-\varphi_Y}\in(0,+\infty)$, and $(f_j)_p\in
\mathcal{O} (K_{\Omega})_{p} \otimes I(\varphi+\Psi)_{p}$ for any $p\in \{z_j\}\times Y$;
	
	$(2)$ $\varphi_{\Omega}+\psi=2\log|g|+2\log|F|$, where $g$ is a holomorphic function on $\{\tilde\Psi<-T\}\cup \tilde Z_3\subset\Omega$ such that $ord_{z_j}(g)=k_j+1$ for any $1\le j<\gamma$;
		
	$(3)$ $\tilde Z_3\not=\emptyset$ and $\psi=2\sum_{1\le j<\gamma}\big(q_{z_j}-ord_{z_j}(F)\big)G_{\Omega_t}(\cdot,z_j)+2\log|F|-t$ on $\Omega_t$ for any $t\ge T$, where $\Omega_t=\{\tilde\Psi<-t\}\cup \tilde Z_3\subset\Omega$  and $G_{\Omega_t}$ is the Green function on $\Omega_t$;
	
	$(4)$ $\frac{q_{z_j}-ord_{z_j}(F)}{ord_{z_j}(g)}\lim_{z\rightarrow z_j}\frac{dg}{a_jw_j^{k_j}dw_j}=c_0$ for any $1\le j<\gamma$, where $c_0\in\mathbb{C}\backslash\{0\}$ is a constant independnet of $j$;
	
	$(5)$ $\sum_{1\le j<\gamma}\big(q_{z_j}-ord_{z_j}(F)\big)<+\infty$.
\end{Proposition}

\begin{proof}
	It follows from Remark \ref{rem:linear} that
 we can assume that $c(t)\ge \frac{e^t}{t^2}$ near $+\infty$. Lemma \ref{lm:trivial module 2} shows that $H_p=I(\varphi+\Psi)_p$ for any $p\in Z_1\backslash Z_3.$

	For any $z_0\in\tilde Z_2$, it follows from Lemma \ref{lm:trivial module 1} and the following remark that $H_p=I(\varphi+\Psi)_p$ for any $p\in\{z_0\}\times Y$.
	\begin{Remark}
		\label{r:not integer}
		Let $b\in(0,1)$. Let $\tilde\varphi_{\Omega}=\varphi_{\Omega}+b(\psi-2\log|F|)$ and $\tilde\psi=(1-b)\psi+2b\log|F|$. Denote that $\tilde\Psi:=\min\{\pi_1^*(\tilde\psi-2\log| F|),(1-b)T\}=(1-b)\Psi$.
		Let $\tilde c(t)=c\left(\frac{t}{1-b}\right)e^{-\frac{bt}{1-b}}$ be a function on $((1-b)T,+\infty)$, and we have $\tilde c(t)\ge 1$ near $+\infty$. It is clear that $\tilde\psi$ and $\tilde\varphi_{\Omega}+\tilde\psi=\varphi_{\Omega}+\psi$ are subharmonic functions. Denote that $\tilde\varphi:=\pi_1^*(\tilde\varphi_{\Omega})+\pi_2^*(\varphi_Y)$.
		
		 Note that $e^{-\tilde\varphi}\tilde c(-\tilde\Psi)=e^{-\pi_1^*(\varphi_{\Omega}+b(\psi-2\log|F|))-\pi_2^*(\varphi_Y)}c(-\Psi)e^{b\Psi}=e^{-\varphi}c(-\Psi)$ on $\{\Psi<-T\}$, $\varphi+\Psi=\tilde\varphi+\tilde\Psi$ on $\{\Psi<-t\}$ and $\tilde c(t)e^{-t}=c\left(\frac{t}{1-b}\right)e^{-\frac{t}{1-b}}$.
		As $z_0\in\tilde Z_2=\{z\in\tilde Z_0:2ord_{z}(F)>v(dd^c(\psi),z)\}$, we can choose $b\in(0,1)$ such that $v(dd^c(\tilde\psi),z_0)+v(dd^c(\tilde\varphi_{\Omega}+\tilde\psi),z_0)=v(dd^c(\psi),z_0)+v(dd^c(\varphi_{\Omega}+\psi),z_0)+b(2ord_{z_0}(F)-v(dd^c(\psi),z_0))\not\in \mathbb{Z}$.
	\end{Remark}
	
Thus, we have $H_p=I(\varphi+\Psi)_p$ for any $p\in Z_0\backslash Z_3.$

It follows from Lemma \ref{characterization of g(t)=0} that there exists a holomorphic $(n,0)$ form $f_{t_1}$ on $\{\Psi<-t_1\}$ such that $(f_{t_1}-f)_{p}\in\mathcal{O}(K_{M})_p\otimes I(\varphi+\Psi)_p$ for any $p\in Z_0$ and $\int_{\{\Psi<-t_1\}}|f_{t_1}|^2e^{-\varphi}c(-\Psi)<+\infty$, which implies that $(f_{t_1})_{p}\in\mathcal{O}(K_{\Omega})_p\otimes H_p$ for any $p\in Z_0$.

Note that there exists a subharmonic function $\psi_1$ on $\Omega_1:= \{\tilde\Psi<-T\}\cup \tilde Z_3$ such that $\psi_1+2\log|F|=\psi$. Denote that $M':=\Omega_1\times Y=\{\Psi<-T\}\cup Z_3\subset M$ and $\tilde\psi:=\pi_1^*(\psi_1)$.   For any $(z_0,y_0)\in  Z_3\backslash\{\Psi<-t_1\}$, we have $F(z_0)=0$. Note that $v(dd^c(\psi),z_0)>2ord_{z_0}(F)$, then $e^{-\varphi}c(-\Psi)$ has a positive lower bound on $(V'\backslash\{z_0\})\times Y_0\subset\{\Psi<-t\}$, where $V'\Subset\Omega_1$ is a neighborhood  of $z_0$ and $Y_0$ is a neighborhood of $y_0$. Following from $\int_{\{\Psi<-t_1\}}| f_{t_1}|^2e^{-\varphi}c(-\Psi)<+\infty$, we get that there exists a holomorphic $(n,0)$ form $\tilde f_{t_1}$ on $\{\tilde\psi<-t_1\}=\{\Psi<-t_1\}\cup Z_3$ such that $\tilde f_{t_1}= f_{t_1}$ on $\{\Psi<-t_1\}$, which implies that  $(\tilde f_{t_1}-f)_p\in\mathcal{O}(K_{\Omega})_p\otimes I(\varphi+\Psi)_p$ for any $p\in Z_0$. By the definition of $G(t;c,\Psi,\varphi,I(\varphi+\Psi),f)$, we have  $G(t;c,\Psi,\varphi,I(\varphi+\Psi),f)=G(t;c,\Psi,\varphi,I(\varphi+\Psi),\tilde f_{t_1})$ for any $t\ge T$. It follows from Lemma \ref{l:inner} and $H_p=I(\varphi+\Psi)_p$ for any $p\in Z_0\backslash Z_3$ that $G_{\tilde f_{t_1}}(t)=G(t;c,\Psi,\varphi,I(\varphi+\Psi),\tilde f_{t_1})$ for any $t\ge T$, which implies that $G_{\tilde f_{t_1}}(h^{-1}(r))$ is linear with respect to $r$. Denote that $\tilde Z'_3:=\{z\in \tilde Z_3:v(dd^c(\varphi_{\Omega}+\psi_1),z)>0\}$ and $Z'_3:=\pi_1^{-1}(\tilde Z'_3)$. Denote
\begin{equation*}
\begin{split}
\inf\Bigg\{ \int_{ \{ \tilde\psi<-t\}}|\tilde{f}|^2&e^{-\varphi}c(-\tilde\psi): \tilde{f}\in
H^0(\{\tilde\psi<-t\},\mathcal{O} (K_{M})  ) \\
&\&\, (\tilde{f}-\tilde f_{t_1},p)\in
\mathcal{O} (K_{M})_{p} \otimes \mathcal{I}(\varphi+\tilde\psi)_p,\text{ for any }  p\in Z'_3 \Bigg\}
\end{split}
\end{equation*}
by $\tilde G(t)$, where $t\in[T,+\infty)$.  For any holomorphic $(n,0)$ form $\hat f$ on $\{\tilde\psi<-t\}$ satisfying $\int_{\{\tilde\psi<-t\}}|\hat f|^2e^{-\varphi}c(-\tilde\psi)<+\infty$ and any $p=(z_0,y_0)\in \tilde Z_3\backslash \tilde Z'_3$, it follows from Lemma \ref{decomp} and Lemma \ref{decomp integrable} that $\hat f|_{U_{z_0}\times Y}=\sum_{j\ge0}\pi_1^*(w^jdw_j)\wedge\pi_2^*(F_j)$, where $w$ is a local coordinate on a neighborhood $V_{z_0}\Subset\{\psi_1<-t\}\subset\Omega_1$ satisfying $w(z_0)=0$ and $U_{z_0}=\{|w(z)|<1\}\Subset V_{z_0}$, $F_j$ is holomorphic $(n-1,0)$ form on $Y$ satisfying $(F_j,y_0)\in(\mathcal{O}(K_Y)\otimes\mathcal{I}(\varphi_Y))_{y_0}$. Note that $\tilde Z_3\backslash \tilde Z'_3=\{z\in \tilde Z_3:v(dd^c(\varphi_{\Omega}+\psi_1),z)=0\}$, then it follows from Lemma \ref{closedness} that
 $(\hat f,p)\in
\mathcal{O} (K_{M})_{p} \otimes \mathcal{I}(\varphi+\tilde\psi)_p$. Hence, we have $\tilde G(h^{-1}(r))=G_{\tilde f_{t_1}}(h^{-1}(r))$ is linear with respect to $r$.

For any $z_0\in\tilde Z'_3$, if $(\psi_1-2q'_{z_0}\log|w|)(z)=-\infty$, where $q'_{z_0}=\frac{1}{2}v(dd^c(\psi_1),z_0)$ and $w$ is a local coordinate on a neighborhood of $z_0$ satisfying that $w(z_0)=0$, we have $(\psi-2q_{z_0}\log|w|)(z_0)=-\infty$ (where $q_{z_0}=\frac{1}{2}v(dd^c(\psi),z_0)$), thus there exists $a\in[0,1)$ such that $\varphi_{\Omega}+a\psi$ is subharmonic near $z_0$. As $v(dd^c(\varphi_{\Omega}+\psi_1),z_0)=v(dd^c(\varphi_{\Omega}+\psi),z_0)-2ord_{z_0}(F)>0$, there exists $a'\in[a,1)$ such that $v(dd^c(\varphi_{\Omega}+a'\psi))-2a'ord_{z_0}(F)>0$, which implies that $\varphi_{\Omega}+a'\psi_1$ is subharmonic near $z_0$.
As $G(t_1;c,\Psi,\varphi,I(\varphi+\Psi),\tilde f_{t_1})=\tilde G(t_1)\in(0,+\infty)$, we have $\tilde Z'_3\not=\emptyset$.  It follows from Proposition \ref{p-finite-2} and Proposition \ref{p-infinite-2} (replace $M$, $\psi$, $Z_0$ and $c(\cdot)$ by $\{\tilde\psi<-t\}=\{\Psi<-t\}\cup Z_3$, $\tilde\psi+t$, $Z'_3$ and $c(\cdot+t)$ respectively, where $t>T$) that the following statements hold:

$(a)$ $\psi_1+t=2\sum_{z\in \tilde Z'_3}q'_zG_{\Omega_t}(\cdot,z)$ on $\Omega_t=\{\psi_1<-t\}$ for any $t>T$;

$(b)$  $\tilde f_{t_1}=\pi_1^*(a_jw_j^{k_j}dw_j)\wedge \pi_2^*(f_Y)+\tilde f_j$ on $V_{z_j}\times Y$ for any $z_j\in \tilde Z'_3$,  where $a_j\in\mathbb{C}\setminus \{0\}$ is a constant, $k_j$ is a nonnegative integer, $f_{Y}$ is a holomorphic $(n-1,0)$ form on $Y$ such that $\int_{Y}|f_{Y}|^2e^{-\varphi_{Y}}\in (0,+\infty)$, and $(\tilde f_j,(z_j,y))\in \mathcal{O}(K_M)_{(z_j,y)}\otimes\mathcal{I}(\varphi+\tilde\psi)_{(z_j,y)}$ for any $z_j\in \tilde Z'_3$ and $y\in Y$;

$(c)$ $\varphi_{\Omega}+\psi_1=2\log|g|$, where $g$ is a holomorphic function on $\Omega_1= \{\Psi<-T\}\cup \tilde Z_3$ such that $ord_z(g)=k_j+1$ for any $z\in \tilde Z'_3$;

$(d)$ $\frac{q'_{z_j}}{ord_{z_j}(g)}\lim_{z\rightarrow z_j}\frac{dg}{a_jw_j^{k_j}dw_j}=c_0$ for any $z_j\in \tilde Z'_3$, where $c_0\in\mathbb{C}\backslash\{0\}$ is a constant independent of $z_j$;

$(e)$ $\sum_{z\in \tilde Z'_3}q'_z<+\infty$.

By definition of $\psi_1$, we know that $q'_z=q_z-ord_{z}(F)$ for any $z\in \Omega_1$.
It follows from $(a)$, Lemma \ref{l:green-sup2} and $\tilde Z_3\in\{z\in\tilde Z_0:v(dd^c(\psi),z)>2ord_z(F)\}$ that $\tilde Z_3=\tilde Z'_3\not=\emptyset$ and $\psi=2\sum_{1\le j<\gamma}\big(q_{z_j}-ord_{z_j}(F)\big)G_{\Omega_t}(\cdot,z_j)+2\log|F|-t$ on $\Omega_t$ for any $t\ge T$. Note that $(\tilde f_{t_1}-f)_p\in\mathcal{O}(K_M)_p\otimes I(\varphi+\Psi)_p$ for any $p\in Z_0$, and $(\tilde{f}_j,p)\in \mathcal{O}(K_M)_{p}\otimes\mathcal{I}(\varphi+\tilde\psi)_{p}$ implies $(\tilde f_j)_p\in \mathcal{O}(K_M)_{p}\otimes {I}(\varphi+\Psi)_{p}$ for any $Z_3$. Let $f_j=f-\pi_1^*(a_jw_j^{k_j}dw_j)\wedge \pi_2^*(f_Y)$ on $(V_{z_j}\times Y)\cap\{\Psi<-t_0\}\cap V$, hence $(f_j)_p=(f-\tilde f_{t_1}+\tilde f_j)_p\in \mathcal{O}(K_M)_p\otimes I(\varphi+\Psi)_p$ for any $p\in\{z_j\}\times Y$. It follows from $(c)$ and $\psi=\psi_1+2\log|F|$ that $\varphi_{\Omega}+\psi=2\log|g|+2\log|F|$.

Thus, the five statements in Proposition \ref{p:n-linearity1} hold.
\end{proof}

\section{Proof of Theorem \ref{thm:fibra-finite}}

 The necessity of Theorem \ref{thm:fibra-finite} follows from Proposition \ref{p:n-linearity1}. In the following, we prove the sufficiency.

As $\psi=2\sum_{1\le j\le m}\big(q_{z_j}-ord_{z_j}(F)\big)G_{\Omega_t}(\cdot,z_j)+2\log|F|-t$ on $\Omega_t$ and $\Psi=\min\{\pi_1^*(\psi-2\log|F|),T\}$, it follows from Lemma \ref{l:G-compact} that
$$Z_0=Z_0\cap(\cap_{s>T}\overline{\{\Psi<-s\}})=\{z_j:j\in\{1,2,\ldots,m\}\}\times Y=Z_3.$$
   It follows from Lemma \ref{characterization of g(t)=0} that there exists a holomorphic $(n,0)$ form $f_{t_1}$ on $\{\Psi<-t_1\}$ such that $(f_{t_1}-f)_{p}\in\mathcal{O}(K_{M})_p\otimes I(\varphi+\Psi)_p$ for any $p\in Z_0$ and $\int_{\{\Psi<-t_1\}}|f_{t_1}|^2e^{-\varphi}c(-\Psi)<+\infty$, which implies that $(f_{t_1})_{p}\in\mathcal{O}(K_{\Omega})_p\otimes H_p$ for any $p\in Z_0$, where $t_1>0$.
Note that there exists a subharmonic function $\psi_1$ on $\Omega_1:= \{\tilde\Psi<-T\}\cup \tilde Z_3$ such that $\psi_1+2\log|F|=\psi$. Denote that $M':=\Omega_1\times Y=\{\Psi<-T\}\cup Z_3\subset M$ and $\tilde\psi:=\pi_1^*(\psi_1)$.   For any $(z_0,y_0)\in  Z_3\backslash\{\Psi<-t_1\}$, we have $F(z_0)=0$. Note that $v(dd^c(\psi),z_0)>2ord_{z_0}(F)$, then $e^{-\varphi}c(-\Psi)$ has a positive lower bound on $(V'\backslash\{z_0\})\times Y_0\subset\{\Psi<-t\}$, where $V'\Subset\Omega_1$ is a neighborhood  of $z_0$ and $Y_0$ is a neighborhood of $y_0$. Following from $\int_{\{\Psi<-t_1\}}| f_{t_1}|^2e^{-\varphi}c(-\Psi)<+\infty$, we get that there exists a holomorphic $(n,0)$ form $\tilde f_{t_1}$ on $\{\tilde\psi<-t_1\}=\{\Psi<-t_1\}\cup Z_3$ such that $\tilde f_{t_1}= f_{t_1}$ on $\{\Psi<-t_1\}$, which implies that  $(\tilde f_{t_1}-f)_p\in\mathcal{O}(K_{\Omega})_p\otimes I(\varphi+\Psi)_p$ for any $p\in Z_0$. Following from the definition of $G_{\tilde f_{t_1}}(t)$ in Section \ref{sec:n},   it follows from Lemma \ref{l:inner} and $Z_0=Z_3$ that $G_{\tilde f_{t_1}}(t)=G(t)$ for any $t\ge T$. Using Proposition \ref{p-finite-2} (replace $M$, $c(\cdot)$ and $\psi$ by $M'$, $c(\cdot+\tilde{t})$ and $\pi_1^*(\psi_1)+\tilde{t}$ respectively, where $\tilde{t}>T$), we know that $G_{\tilde f_{t_1}}(h_{\tilde{t}}^{-1}(r)+\tilde{t})$ is linear with respect to $r\in (0,\int_{0}^{+\infty}c(s+\tilde{t})e^{-s}ds)$ for any $\tilde{t}> T$, where $h_{\tilde{t}}(t)=\int_t^{+\infty}c(s+\tilde{t})e^{-s}ds=e^{\tilde{t}}\int_{t+\tilde{t}}^{+\infty}c(s)e^{-s}ds=e^{\tilde{t}}h(t+\tilde{t})$. Note that $G(h^{-1}(r))=G_{\tilde f_{t_1}}(h^{-1}(r))=G_{\tilde f_{t_1}}(h_{\tilde{t}}^{-1}(e^{\tilde{t}}r)+\tilde{t})$ for any $r\in(0,\int_{\tilde{t}}^{+\infty}c(s)e^{-s}ds)$. Hence we have $G(h^{-1}(r))$ is linear with respect to $r\in (0,\int_{0}^{+\infty}c(s)e^{-s}ds)$.

Thus, Theorem \ref{thm:fibra-finite} holds.


\vspace{.1in} {\em Acknowledgements}. The first author and the second author were supported by National Key R\&D Program of China 2021YFA1003103.
The first author was supported by NSFC-11825101, NSFC-11522101 and NSFC-11431013. The authors would like to thank Shijie Bao for checking the manuscript and point out some typos.

\bibliographystyle{references}
\bibliography{xbib}

\end{document}